\documentclass[11pt]{article}
\usepackage[a4paper,top=3cm,bottom=3.5cm,left=2.5cm,right=2.5cm,marginparwidth=1.75cm]{geometry}
\usepackage[english]{babel}
\usepackage[utf8x]{inputenc}
\usepackage[T1]{fontenc}
\usepackage{graphicx}
\usepackage{caption, subcaption}
\usepackage[title]{appendix}
\usepackage{tikz}
\usepackage{pgfplots}
\usepackage[super]{nth}
\usepackage{cutwin}
\graphicspath{ {images/} }
\usetikzlibrary{intersections,decorations.pathreplacing,decorations.markings,calc,angles,quotes,arrows.meta,pgfplots.fillbetween,patterns}

\usepackage{amsmath,yhmath}
\usepackage{amsthm}
\usepackage{amssymb}
\usepackage{hyperref}
\usepackage{cleveref}
\usepackage{epstopdf}
\usepackage{comment}
\usepackage{todonotes}
\usepackage{color}
\usepackage{float}
\usepackage[sort&compress,numbers]{natbib}
\newtheorem{thm}{Theorem}
\newtheorem{lem}[thm]{Lemma}

\theoremstyle{definition}
\newtheorem{defn}[thm]{Definition}
\newtheorem{example}[thm]{Example}

\newtheorem{rmk}[thm]{Remark}
\numberwithin{equation}{section}
\numberwithin{thm}{section}

\newcommand{\N}{\mathbb{N}}

\newcommand{\R}{\mathbb{R}}

\renewcommand{\S}{\mathcal{S}}

\newcommand{\Z}{\mathbb{Z}}

\newcommand{\p}{\partial}
\newcommand{\dx}{\mathrm{d}}

\newcommand{\nm}{\noalign{\smallskip}}
\newcommand{\ds}{\displaystyle}
\newcommand{\iu}{\mathrm{i}\mkern1mu}

\makeatletter
\newcommand{\neutralize}[1]{\expandafter\let\csname c@#1\endcsname\count@}
\makeatother
\title{Unidirectional edge modes in time-modulated metamaterials\thanks{\footnotesize
This work was supported in part by the Swiss National Science Foundation grant number
200021--200307.}}

\author{
Habib Ammari\thanks{\footnotesize Department of Mathematics, 
	ETH Z\"urich, 
	R\"amistrasse 101, CH-8092 Zurich, Switzerland (habib.ammari@math.ethz.ch, jinghao.cao@sam.math.ethz.ch).} \and Jinghao Cao\footnotemark[2]}

\date{}

\begin{document}
	\maketitle
	
	\begin{abstract}
	We prove the possibility of achieving unidirectional edge modes in time-modulated supercell structures. Such finite structures consist of two trimers repeated periodically. Because of their symmetry, they admit degenerate edge eigenspaces. When the trimers are time-modulated with two opposite orientations, 
	the degenerate eigenspace splits into two one-dimensional eigenspaces described by an analytical formula, each corresponds to a mode which is localized at one edge of the structure. Our results on the localization and stability of these edge modes with respect to fluctuations in the time-modulation amplitude are illustrated by several numerical simulations. 
	\end{abstract}

\noindent{\textbf{Mathematics Subject Classification (MSC2000):} 35J05, 35C20, 35P20, 74J20
		
\vspace{0.2cm}
		
\noindent{\textbf{Keywords:}} time-modulated metamaterial, valley Hall effect, unidirectional edge mode, supercell structure, artificial spin effect}
	\vspace{0.5cm}

 
\section{Introduction}

The work presented in this paper aims at contributing to the mathematical and numerical analysis of wave manipulation in metamaterials. Metamaterials are made of  high-contrast subwavelength resonant unit cells, each of them is interacting strongly with low-frequency incoming waves due to subwavelength resonances \cite{lemoult2016soda,yves2017crytalline,
phononic1,phononic2}. Their study is motivated by a number of potential technologically relevant applications concerning wave manipulation at deep subwavelength scales such as subwavelength confinement and guiding of waves and superresolution sensing \cite{ammari2017subwavelength,yves2017crystalline,yves2017topological,wang2019subwavelength}. 
It has led to the discovery of classical wave analogies of important behaviors observed in condensed matter physics. Typical examples include analogies to topologically protected edge modes, unidirectional transport phenomenon, exceptional points and Dirac degeneracies, and the quantum valley Hall effect  \cite{ammari2021functional,review,review2,topo1}.

Recently, many new fundamental mathematical results in the theory of wave propagation in complex resonant structures have been obtained and  efficient and  sophisticated boundary element techniques for modeling wave control and manipulation in metamaterials have been designed, see e.g. the review paper \cite{ammari2021functional} and the dissertations \cite{erik_thesis,bryn_thesis}.
   
The  quantum analogue of high-contrast subwavelength metamaterials (in condensed matter physics) is  graphene, which is a single atomic layer of carbon atoms arranged in a honeycomb structure. While Fefferman, Weinstein and their collaborators \cite{drouot2, drouot1,fefferman,fefferman2,fefferman3,fefferman4, lee-thorp,drouot_valley} proved fundamental mathematical results on the graphene model, little was known  on the classical analogue (in wave physics) of this. This is due to many fundamental differences in the mathematical treatments of the classical and quantum problems. The results obtained on graphene are all based on the use of the tight-bending model for Schr\"odinger operators (which is a discrete approximation of the continuous model) together with the nearest-neighbour
approximation, which leads to tridiagonal matrix approximations to Schr\"odinger operators. A main consequence of this is that chiral symmetry holds, which makes the proof of existence and the characterization of the topologically protected edge modes (i.e., localized modes that are robust with respect to random changes in the geometry and/or  material parameters of the system) tractable. 

In wave physics, due to the strong interactions between the
subwavelength resonators (or in other terms because of the slow decay of the solutions to the corresponding partial differential equations), there are fundamental differences compared to the quantum case in the analysis of the model problems, in particular their topological properties.
Moreover, analytical and numerical schemes for finding the properties of high-contrast subwavelength metamaterials would lead to incorrect modeling if there are based on the tight-bending model \cite{fiorani}. Nevertheless, as shown in \cite{ammari2021functional,ammari2017double,bryn_thesis,erik_thesis}, a discrete approximation of the corresponding partial differential equations  does exist. It is in terms of generalized capacitance matrices. The strength of such novel approximate formulation is
that the capacitance matrix accounts for these strong interactions, thereby providing an efficient numerical approach and a unified mathematical
model to study challenging problems in subwavelength wave physics. This approach allows to establish for the first time a complete mathematical and numerical framework for metamaterials: (i) a discrete approximation for computing subwavelength resonances in both finite and periodic systems of subwavelength resonators has been introduced \cite{davies2019fully,ammari2021functional,ammari2017double}; (ii) the existence of Dirac singularities was proved \cite{ammari2020honeycomb,ammari2020highfrequency}; (iii) the role that Dirac points play in the origin of robust (topologically protected) edge states has been explored \cite{ammari2020topological}, and (iv) it has been shown that the introduction of time-modulations into systems of subwavelength resonators can shift the Dirac singularities to the origin of the Brillouin zone \cite{ammari2020time}. This approach has been also extended to the analysis of exceptional points, non-hermitian systems of subwavelength resonators and their applications in sensing at subwavelength scales \cite{ammari2020highorder,ammari2020exceptional,ammari2020edge}. 

To the best of our knowledge,  the mathematical and numerical analysis of wave propagation properties of time-modulated metamaterials has just been started.  
In \cite{ammari2020time}, a discrete characterization of the
band structure in time-modulated metamaterials is introduced. This characterization provides both theoretical insight and an efficient numerical
method to compute the dispersion relationship of time-modulated systems of subwavelength resonators. 
A study of exceptional points in the case of time-modulated  metamaterials is presented in 
\cite{TheaThesis}. Furthermore, in \cite{paper1} the possibility of achieving non-reciprocal wave propagation in time-modulated metamaterials is shown. Finally, in \cite{Thea2}, the question whether an analogous principle as the bulk-boundary correspondence of quantum systems is possible in the case of  time-modulated metamaterials is discussed. 

In the present paper, we are concerned with the mathematical foundation of the (classical) analogue of the  valley Hall effect \cite{valley1, valley2, valley3, valley4, valley6, valley7}  for lattices of time-modulated subwavelength resonators. 
We consider supercell structures where phase-shifted (``rotation like'') time-modulations of subwavelength resonators can provide a kind of ``artificial spin''. 
We prove that the valley Hall effect occurs in truncated  supercell lattices of time-modulated subwavelength resonators by opening non-reciprocal band gaps at degenerate points. We give evidence that unidirectional edge localization phenomena are not particular to quantum systems, as conjectured in the seminal papers \cite{haldane2,haldane}.  For doing so, we start by showing that the capacitance matrix associated with a supercell structure has a degenerate eigenspace of dimension two. The associated eigenvectors correspond to two edge modes of the unmodulated structure. Then we consider a time-modulated structure consisting of consecutive trimers  modulated in opposite orientations to replicate spin effects from quantum systems. As a consequence of this time-modulation which breaks the time-reversal symmetry of the structure, the degenerate eigenspace splits into two one-dimensional spaces, each corresponds to a mode localized at one of the edges of the modulated structure. We also verify numerically that these two unidirectional edge modes are stable with respect to fluctuations in the modulation amplitude. 
It is worth emphasizing that in  \cite{ammari2020highfrequency,ammari2020honeycomb}, it is proved that honeycomb lattices of (unmodulated) subwavelength  resonators support degenerate singularities of Dirac type.  This degeneracy holds also for the chains of supercells considered in this paper. By breaking of time-reversal symmetry by modulating the resonators, we show that valley Hall effect can be realized in supercell structures. In other words, such spatiotemporally modulated structures enable then unidirectional edge modes and posses unique features in terms of wave localization. Our results in this paper also allow to envision the mathematical foundation of unidirectional wave guiding phenomena in (biperiodic) screens of space-time modulated subwavelength resonator systems.

This paper is organized as follows. In Section 2, we formulate the problem of reciprocity and review the Floquet-Bloch theory which is essential to solve ordinary differential equations with periodic coefficients. In Section 3, we provide an approximation of the time-dependent Bloch modes inside the resonators. Using this approximation formula in the unmodulated case, we reproduce the  numerical results of \cite{ammari2020topological} on the existence and localization of a topologically protected edge mode in a Su-Schrieffer-Heeger (SSH) chain of subwavelength resonators. 
Section 4 is devoted to the localization of edge modes in a supercell structure. We first determine algebraically the corresponding capacitance matrix. Then, considering the dilute regime, we show the approximate degeneracy of at least one of the eigenvalues of the capacitance matrix. Such a degenerate eigenvalue yields a degenerate edge eigenspace for the supercell structure. The edge modes are localized both at the left and right edges of the structure. Then by time modulating the supercell structure in such a way that the two adjacent trimers have different orientations, we show both analytically and numerically that the degenerate eigenspace splits into two one-dimensional eigenspaces, each corresponds to an edge mode localized either at the left or at the right of the structure.  We also illustrate numerically the stability of such unidirectional edge modes  
with respect to random  fluctuations of the time-modulation amplitude.

\section{Problem formulation and preliminary theory}	
In this section, we formulate the problem of study. Moreover, we introduce the Floquet-Bloch theory for periodic differential equations. This subsection follows closely the introductory theory provided in \cite{ammari2020time}.

\subsection{Resonator structures and the wave equation}
\label{sec:formulation}
We consider the wave equation in structures with time-modulated materials. Such wave equation can be used to model acoustic and polarized electromagnetic waves.
The time dependent material parameters are given by $\rho(x,t)$ and $\kappa(x,t)$. 
In acoustics, $\rho$ and $\kappa$ represent the density and the bulk modulus of the materials. We study the wave equation with time-dependent coefficients in dimension three:
\begin{equation}
\label{waveequation}
\left(\frac{\partial}{\partial t }\frac{1}{\kappa(x,t)}\frac{\partial}{\partial t}-\nabla\cdot\frac{1}{\rho(x,t)}\nabla\right)u(x,t)=0,\ \ x\in \mathbb{R}^3,t\in\mathbb{R}.	
\end{equation}

Let $D$ denote a system of $N$ resonators. $D$ is constituted by $N$ disjoint domains $D_i$ for $i=1,\ldots,N$, each $D_i$ being connected and having boundary of H\"older class $\p D_i \in C^{1,s}, 0 < s < 1$.

For the purpose of this paper, we apply time-modulations to the interior of the resonators, while the surrounding material is constant in $t$. We let
\begin{equation}
	\label{modulation}
	\kappa(x,t)=
\begin{cases}
	\kappa_0, \ & x\in\mathbb{R}^3\backslash \overline{D}\\
	\kappa_r\kappa_i(t),\ & x\in D_i,
\end{cases} 
,\ \ 
	\rho(x,t)=
\begin{cases}
	\rho_0, \ & x\in\mathbb{R}^3\backslash \overline{D},\\
	\rho_r\rho_i(t),\ & x\in D_i,
\end{cases},
\end{equation}
for $i=1,\ldots, N$. Here, $\rho_0$, $\kappa_0$, $\rho_r$, and $\kappa_r$ are positive constants. The functions $\rho_i(t) \in \mathcal{C}^0(\mathbb{R})$ and $\kappa_i(t) \in \mathcal{C}^1(\mathbb{R})$ describe the modulation inside the $i^{\text{th}}$ resonator $\mathcal{C}_i$. We assume that each of the functions $\rho_i$ and $\kappa_i$ is periodic with period $T$.

We define the contrast parameter $\delta$ as 
$$
\delta := \frac{\rho_r}{\rho_0}.
$$
In (\ref{waveequation}), we have the transmission conditions at $x\in \p D_i$
$$u \big|_+ = u\big|_- \quad \mbox{and} \quad \delta \frac{\partial {u}}{\partial \nu} \bigg|_{+} - \frac{1}{\rho_i(t)}\frac{\partial {u}}{\partial \nu} \bigg|_{-} = 0, \qquad x\in \p D_i, \ t\in \R,$$
where $\partial/\partial \nu$ is the outward normal derivative at $\p D_i$ and 
the subscripts $+$ and $-$ denote taking the limit from outside and inside $D_i$, respectively.

In order to achieve subwavelength resonances, we assume that $\delta \ll 1$ and consider the regime where the modulation frequency $$\Omega := \frac{2\pi}{T} = O(\delta^{1/2}).$$ 
We also assume that $\dx \kappa_i/\dx t= O(\delta^{1/2})$ for $i=1,\ldots,N.$

Note that in the static case where there is no modulation of the material parameters (i.e., when $\rho_i(t)=\kappa_i(t)=1$ for all $i$), the system of $N$ subwavelength resonators has $N$ {\em subwavelength frequencies} of order of $O(\delta^{1/2})$. We refer the reader to \cite{davies2019fully,ammari2021functional} for the details.

\subsection{Capacitance matrix}
We first introduce the fundamental solution to the Laplacian 
$$
G(x,y):= - \frac{1}{4 \pi |x-y|} \quad \mbox{for } x \neq y. 
$$
Let $D\subset \R^3$ be as in Section \ref{sec:formulation}. We define the single-layer potential $\mathcal{S}_D: L^2(\partial D) \rightarrow H_{\textrm{loc}}^1(\R^3)$ by
$$\mathcal{S}_D[\phi](x) := \int_{\partial D} G(x,y) \phi(y) \dx\sigma(y),\quad x\in \mathbb{R}^3.$$
Here, the space $H_{\textrm{loc}}^1(\R^3)$ consists of functions that are square integrable and with a square integrable weak first derivative on every compact subset of $\R^3$. Taking the trace on $\p D$, it is well-known that $\S_D: L^2(\p D) \rightarrow H^1(\p D)$ is invertible \cite{MaCMiPaP}. 

\begin{defn}[Capacitance matrix]
	The capacitance coefficients $C_{i,j}$ are defined as
\begin{equation}\label{eq:psiC}
	\psi_i = \left(\S_D \right)^{-1}[\chi_{\p D_i}], \qquad C_{i,j}= -\int_{\p D_i} \psi_j  \dx \sigma,
\end{equation}
for $i,j=1,\ldots,N$, where $\chi_{\p D_i}$ is the characteristic function of $\partial D_i$. The capacitance matrix $C$ is defined as the matrix $C = \left(
C_{i,j}\right)_{i,j=1}^N$.
\end{defn}

\subsection{Floquet-Bloch theory and asymptotic Floquet matrix elements}
\label{sec:Floquettheory}
	Let $A(t)$ be a $T$-periodic $N\times N$ complex matrix function and consider the following ordinary differential equation (ODE):
	\begin{equation}
	\label{ode}
		\frac{\dx x}{\dx t}(t) =A(t)x(t).
	\end{equation}
Recall that the fundamental solution matrix of (\ref{ode}) is a $N\times N$ matrix with linearly independent column vectors, which solves (\ref{ode}). The following theorem is classical (see, for instance, \cite{teschl2012ordinary}). 
\begin{thm}
	(Floquet's theorem) Denote $X(t)$ the matrix-valued fundamental solution with initial value $X(0)=\mathrm{Id}_N$, where $\mathrm{Id}_N$ is the $N\times N$ identity matrix. There exists a constant matrix $F$ and a $T$-periodic matrix function $P(t)$ such that
	\begin{equation}
	\label{floquetthm}
	X(t)= P(t) e^{Ft}.
	\end{equation}
\end{thm}
\noindent
For each eigenvalue $\lambda := e^{\mathrm{i}\omega}$ of $e^{F}$, there is a Bloch solution $x(t)$ which is $\omega$-quasiperiodic, i.e., it satisfies
\begin{equation*}
	x(t+T)=e^{\iu\omega T}x(t).
\end{equation*}
Observe that $\omega$ is defined modulo  the modulation frequency $\Omega$. Therefore, we define the time-Brillouin zone as $Y_t^*:=\mathbb{C}/(\Omega\mathbb{Z})$.

\begin{rmk}
	In some literature, e.g. \cite{Yakubovich}, $e^{\mathrm{i}\omega T}$ is called  a characteristic multiplier. Here, we refer to $\omega$ as a quasifrequency and call $\iu \omega$ a Floquet exponent.
\end{rmk}

If the matrix $A$ is time-independent, then the solution to \eqref{ode} can be written as $x(t)=e^{At}x(0)$. The Floquet exponents are then given by the eigenvalues of $A$. Since the Floquet exponents are defined modulo $\iu\Omega$, we need the following definition.
\begin{defn}[Folding number]
\label{foldingnumber}
	Let $\omega_A$ be the imaginary part of an eigenvalue of the time independent matrix $A$. Then we can uniquely write $\omega_A=\omega_0+m\Omega$, where $\omega_0\in [-\Omega/2,\Omega/2)$. The integer $m$ is called the folding number.
\end{defn}
\begin{rmk}
    In other words, if $A$ is diagonal and constant in time, the solution matrix can be written as 
    \begin{equation*}
        X(t)=e^{At} = e^{Ft}.
    \end{equation*} 
    We choose the imaginary part of $F$ in such a way that $\mathrm{Im}(F_{l,l})\in[-\Omega/2,\Omega/2)$ for $l=1,\ldots, N$, where the $F_{l,l}$'s are the diagonal entries of $F$. 
\end{rmk}

In \cite{TheaThesis}, it is shown that if $A(t)$ is an analytical function of some parameter $\varepsilon$ at $\varepsilon=0$ and has the following expansion:
\begin{equation*}
    A(t)=A_0 +\varepsilon A_1(t) + \varepsilon A_2(t)+ \ldots,
\end{equation*}
where $A_0$ is constant in $t$ and diagonal and the series is convergent for $|\varepsilon| < r_0$ with $r_0>0$ being independent of $t$,  
then the Floquet matrix $F$ is also an analytic function of $\varepsilon$ at $\varepsilon=0$ and can be extended as follows: 
\begin{equation} \label{exp2.8}
    F=F_0+\varepsilon F_1 +\varepsilon^2 F_2 + \ldots.
\end{equation} 
Moreover, the matrix elements of $F_1$ and some entries of $F_2$ are given by the following theorem.
\begin{thm}
\label{thm:expansionFloquet}
Assume that $A_0$ is constant in $t$ and diagonal. Then the first-order term in  the expansion (\ref{exp2.8}) of the Floquet matrix $F$ is given by 
\begin{equation}
\label{eq:firstorderFloquetmatrix}
\left(F_{1}\right)_{k,l}= \begin{cases}\left(A_{1}^{n_{k}-n_{l}}\right)_{k,l} & \text { if }\left(F_{0}\right)_{k,k}=\left(F_{0}\right)_{l,l}, \\ 
\nm
\ds \left(\left(F_{0}\right)_{l,l}-\left(F_{0}\right)_{k,k}\right) \sum_{m \in \mathbb{Z}} \frac{\left(A_{1}^{m}\right)_{k,l}}{\frac{2 \pi \mathrm{i}}{T} m+\left(A_{0}\right)_{l,l}-\left(A_{0}\right)_{k,k}} & \text { if }\left(F_{0}\right)_{k,k} \neq\left(F_{0}\right)_{l,l}.\end{cases}    
\end{equation}
Furthermore,  if $(F_0)_{k,k}=(F_0)_{l,l}$, then  the entries  $(F_2)_{k,l}$ of  the second-order term in (\ref{exp2.8}) read:
\begin{equation}
\label{eq:secondorderFloquetmatrix}
\begin{split}
\left(F_{2}\right)_{k,l}=&\sum_{j=1}^{N} \sum_{m \neq n_{j}-n_{l}} \frac{\left(A_{1}^{n_{k}-n_{l}-m}-A_{1}^{n_{k}-n_{j}}\right)_{k,j}\left(A_{1}^{m}\right)_{j,l}}{\frac{2 \pi \mathrm{i}}{T} m+\left(A_{0}\right)_{l l}-\left(A_{0}\right)_{j,j}}\\
&+\sum_{j=1}^{N} \sum_{m \neq n_{k}-n_{j}} \frac{\left(A_{1}^{m}\right)_{k,j}\left(F_{1}\right)_{j,l}}{\frac{2 \pi \mathrm{i}}{T} m+\left(A_{0}\right)_{j,j}-\left(A_{0}\right)_{k,k}}+\left(A_{2}^{n_{k}-n_{l}}\right)_{k,l}.    
\end{split}
\end{equation}
Here, $n_k$ denotes the folding number of the $k$th eigenvalue of $A_0$ in the sense of Definition \ref{foldingnumber}. The superscripts of $A_1$ and $A_2$ represent the corresponding Fourier coefficients in $t$. For example, $A_1^m$ refers to the $m$th Fourier coefficients of $A_1^m(t)$ in $t$. 
\end{thm}

\subsection{Time-modulated subwavelength resonators}
Seeking quasiperiodic solutions in $t$, we obtain the differential problem
 \begin{equation} \label{eq:wave_transf}
 	\begin{cases}\ \ds \left(\frac{\p }{\p t } \frac{1}{\kappa(x,t)} \frac{\p}{\p t} - \nabla \cdot \frac{1}{\rho(x,t)} \nabla\right) u(x,t) = 0,\\[0.3em]
 		\	u(x,t)e^{-\iu \omega t} \text{ is $T$-periodic in $t$}. 
 	\end{cases}
 \end{equation} 
We seek $\omega\in Y_t^*$ such that there is a non-zero solution $u$ to (\ref{eq:wave_transf}), where the time-modulations are given by  (\ref{modulation}).

Since $e^{-\mathrm{i}\omega t}u(x,t)$ is a $T$-periodic function of $t$, we can write its Fourier series as
$$u(x,t)= e^{\iu \omega t}\sum_{n = -\infty}^\infty v_n(x)e^{\iu n\Omega t}.$$
In the frequency domain, we then have from (\ref{eq:wave_transf}) the following equation, for $n\in \Z$:
\begin{equation} \label{eq:freq}
	\left\{
	\begin{array} {ll}
		\ds \Delta {v_n}+ \frac{\rho_0(\omega+n\Omega)^2}{\kappa_0} {v_n}  = 0 & \text{in } \R^3 \setminus \overline{D}, \\[0.3em]
		\ds \Delta v_{i,n}^* +\frac{\rho_r(\omega+n\Omega)^2}{\kappa_r} v_{i,n}^{**}  = 0 & \text{in } D_i, \\
		\nm
		\ds  {v_n}|_{+} -{v_n}|_{-}  = 0  & \text{on } \partial D, \\
		\nm
		\ds  \delta \frac{\partial {v_n}}{\partial \nu} \bigg|_{+} - \frac{\partial v_{i,n}^* }{\partial \nu} \bigg|_{-} = 0 & \text{on } \partial D_i. \\[0.3em]
	\end{array}
	\right.
\end{equation}
Here, for $i=1,\ldots, N$, $v_{i,n}^*(x)$ and $v_{i,n}^{**}(x)$ are defined through the convolutions
$$v_{i,n}^*(x) = \sum_{m = -\infty}^\infty r_{i,m} v_{n-m}(x), \quad  v_{i,n}^{**}(x) = \frac{1}{\omega+n\Omega}\sum_{m = -\infty}^\infty k_{i,m}\big(\omega+(n-m)\Omega\big)v_{n-m}(x),$$
where $r_{i,m}$ and $k_{i,m}$ are the Fourier series coefficients of $1/\rho_i$ and $1/\kappa_i$, respectively:
$$\frac{1}{\rho_i(t)} = \sum_{n = -\infty}^\infty r_{i,n} e^{\iu n \Omega t}, \quad \frac{1}{\kappa_i(t)} = \sum_{n = -\infty}^\infty k_{i,n} e^{\iu n \Omega t}.$$
We can assume that the solution is normalized as $\|v_0\|_{H^1(Y)} = 1$. Since $u$ is continuously differentiable in $t$, we then have as $n\to \infty$,
\begin{equation} \label{eq:reg_v}
	\|v_n\|_{H^1(Y)} = o\left(\frac{1}{n}\right).
\end{equation}

We will consider the case when the time-modulations of $\rho$ and $\kappa$ consist of a finite Fourier series with a large number of nonzero Fourier coefficients: 
$$\frac{1}{\rho_i(t)} = \sum_{n = -M}^M r_{i,n} e^{\iu n \Omega t}, \qquad \frac{1}{\kappa_i(t)} = \sum_{n = -M}^M k_{i,n} e^{\iu n \Omega t},$$
for some fixed $M\in \N$. We seek subwavelength quasifrequencies $\omega$ of the wave equation \eqref{eq:wave_transf} in the sense of the following definition introduced in \cite{ammari2020time}.
 \begin{defn}[Subwavelength quasifrequency] \label{def:sub}
 		A quasifrequency $\omega = \omega(\delta) \in Y^*_t$ of \eqref{eq:wave_transf} is said to be a \emph{subwavelength quasifrequency} if there is a corresponding Bloch solution $u(x,t)$, depending continuously on $\delta$, which is essentially supported in the low-frequency regime, i.e., it can be written as
 		$$u(x,t)= e^{\iu \omega t}\sum_{n = -\infty}^\infty v_n(x)e^{\iu n\Omega t},$$
 		where, as $\delta \to 0$, 
 		$$\omega \rightarrow 0 \quad \mbox{and} \quad \sum_{n = -\infty}^\infty \|v_n\|_{L^2(Y)} = \sum_{n = -M}^M \|v_n\|_{L^2(Y)} + o(1).$$
 	\end{defn}
One can prove  that the subwavelength quasifrequency $\omega$ and the frequency of modulation $\Omega$ have the same order:
$$\omega = O\left(\delta^{1/2}\right).$$ 	

The following is a capacitance matrix characterization of the subwavelength resonant frequencies  of time-dependent periodic systems of subwavelength resonators. 
\begin{thm}[\cite{ammari2020time}] \label{thm:pre}
		As $\delta \to 0$, the subwavelength quasifrequencies of the wave equation (\ref{eq:wave_transf}) are, to leading order, given by the quasifrequencies of the system of ODEs:
		\begin{equation}
		\label{Hill}
			\frac{\dx^2\Psi}{\dx t^2}(t)+M(t)\Psi(t)=0,	
		\end{equation}
		where $M$ is the matrix defined as
		\begin{equation}
		\label{eq:strutureM}
			M(t)=\frac{\delta\kappa_r}{\rho_r}W_1(t)C W_2(t)+W_3(t)
		\end{equation}
		with $W_1,W_2,$ and $W_3$ being the diagonal matrices with diagonal entries
		\begin{equation}
			(W_1)_{i,i}=\frac{\sqrt{\kappa_i}\rho_i}{\lvert D_i\rvert},\quad (W_2)_{i,i}=\frac{\sqrt{\kappa_i}}{\rho_i},\quad (W_3)_{i,i}=\frac{\sqrt{\kappa_i}}{2}\frac{\dx}{\dx\text{t}}\frac{\dx\kappa_i/\dx t}{\kappa_i^{3/2}}.
		\end{equation}
\end{thm}

\begin{example}
Throughout this paper, we work with the following physical meaningful modulations of $\rho(t)$ and $\kappa(t)$:
\begin{equation}
    \label{modulation_parameters}
    \begin{cases}
        &\ds \rho_i(t) =\frac{1}{1+\varepsilon \mathrm{cos}(\Omega t+\phi_i)} \, ,\\
        \nm
        &\ds \kappa_i(t)\ \textrm{constant}, 
    \end{cases}
\end{equation}
for $i=1,\ldots, N$, where $\varepsilon$ is the modulation amplitude and 
the $\phi_i$'s are phase shifts. 

Inserting (\ref{modulation_parameters}) into (\ref{eq:strutureM}), we find the entries of the coefficient matrix $M$ in the Hill's equation to be
\begin{equation}
    \label{eq:expansionMij}
    M_{i,j}= (M_0)_{i,j} (1+\varepsilon^k((-1)^{k+1}\mathrm{cos}^{k-1}(\Omega t+\phi_j)(\mathrm{cos}(\Omega t+\phi_i))-\mathrm{cos}(\Omega t+\phi_j))
\end{equation}
with
\begin{equation}
    M_0 := \frac{\delta \kappa_r}{\rho_r}C.
\end{equation}
\end{example}
\section{Time-dependent Bloch modes}
We want to find $\omega = O(\sqrt{\delta})$ for $\delta$ small enough such that there is a non-zero solution $u$ to
\begin{equation}
\left\{\begin{array}{l}
\left(\frac{\partial}{\partial t} \frac{1}{\kappa(x, t)} \frac{\partial}{\partial t}-\nabla \cdot \frac{1}{\rho(x, t)} \nabla\right) u(x, t)=0, \\
\nm
u(x, t) e^{-\mathrm{i} \omega t} \text { is } T \text {-periodic in } t.
\end{array}\right.
\end{equation}
We  refer to $u: \mathbb{R}^3\times \mathbb{R}\rightarrow \mathbb{C}$ as a Bloch eigenmode. Following \cite{ammari2020time}, we expand $u$ in the form 
$$
u(x, t)=e^{\mathrm{i} \omega t} \sum_{n=-\infty}^{\infty} v(x, n) e^{\mathrm{i} n \Omega t}.
$$
We introduce 
$$
V_{i}(n)=\int_{D_i}v(x,n) \mathrm{d}x
\quad \mbox{and} \quad
V_{i}(t)=e^{\mathrm{i}\omega t}\sum_{-\infty}^{\infty}V_i(n)e^{\mathrm{i}n\Omega t}. 
$$
Then, it follows from  \cite{ammari2020time} that 
$$\Psi=\left(\frac{\rho_{i}(t)}{\sqrt{\kappa_{i}(t)}} c_{i}(t)\right)_{i=1}^{N} \quad \mbox{with} \quad  c_i(t)=\frac{V_i(t)}{|D_i|\rho_i(t)}
$$
solves the Hill equation (\ref{Hill}). Moreover, we have that
\begin{equation}
\label{Blochmodes}
\int_{D_i}u(x,t)\mathrm{d}x=V_i(t)=c_i(t)|D_i|\rho_i(t)=\Psi_i(t)\frac{|D_i|}{\sqrt{\kappa_i(t)}}.    
\end{equation}
Furthermore, we can prove that the Bloch mode $u$ is approximately constant in $x$ in each resonators (as $\delta$ goes to zero), i.e., 
\begin{equation}
\label{eq:approxiamtionbloch}
    u(x,t)|_{D_i}\approx \frac{1}{|D_i|}\int_{D_i}u(x,t)\mathrm{d}x.
\end{equation}\\
\begin{thm}
\label{thm:blochapprox}
	In the subwavelength regime, the Bloch eigenmode restricted to the resonators is approximately given by
\begin{equation}
\label{eq:finiteHill}
    u(x,t)|_{D_i} \approx \frac{\Psi_i(t)}{\sqrt{\kappa_i}}.
\end{equation}
\end{thm}

In the following example, we use formula (\ref{eq:finiteHill}) to simulate edge modes in the static case (i.e., the unmodulated case) 
in a SSH chain of resonators. 
This structure is based on the intuition that if one joins together two chains with different topological properties, a protected edge mode will occur at the interface (this is the principle of bulk-boundary correspondence). Our formula numerically shows as in \cite{ammari2020topological} that this chain exhibits a topologically protected subwavelength edge mode.

%
%
\begin{example} \label{example_topo}
In \cite{ammari2020topological}, it is proven that the (topologically nontrivial) structure illustrated in Figure \ref{fig:SSH13_structure} has an edge mode due to a geometrical defect, resulting from the Zak phase differences.  In concrete, this effect is observed in structures where a trimer is repeated finitely many times with a single resonator in the middle, as depicted in Figure \ref{fig:SSH}. Using the approximation formula (\ref{eq:finiteHill}) in Theorem \ref{thm:blochapprox} we can obtain the same results showing existence of the edge Bloch modes; see Figure \ref{fig:SSH13_edge}. Note that the numerical simulations presented in  \cite{ammari2020topological} were based on solving the underlying partial differential equations ans not on using the approximation formula (\ref{eq:finiteHill}).

\begin{figure}[H]
\centering
    \begin{subfigure}[t]{0.45\textwidth}
        \includegraphics[scale=0.5]{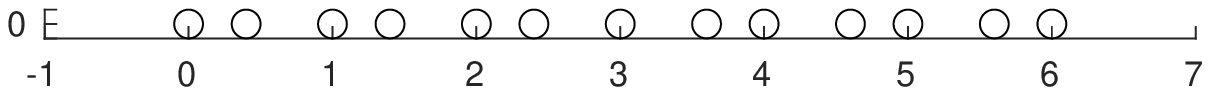}
        \caption{A chain with 13 resonators.}
        \label{fig:SSH13_structure}
    \end{subfigure}
\begin{subfigure}[t]{0.45\textwidth}
    \includegraphics[scale=0.5]{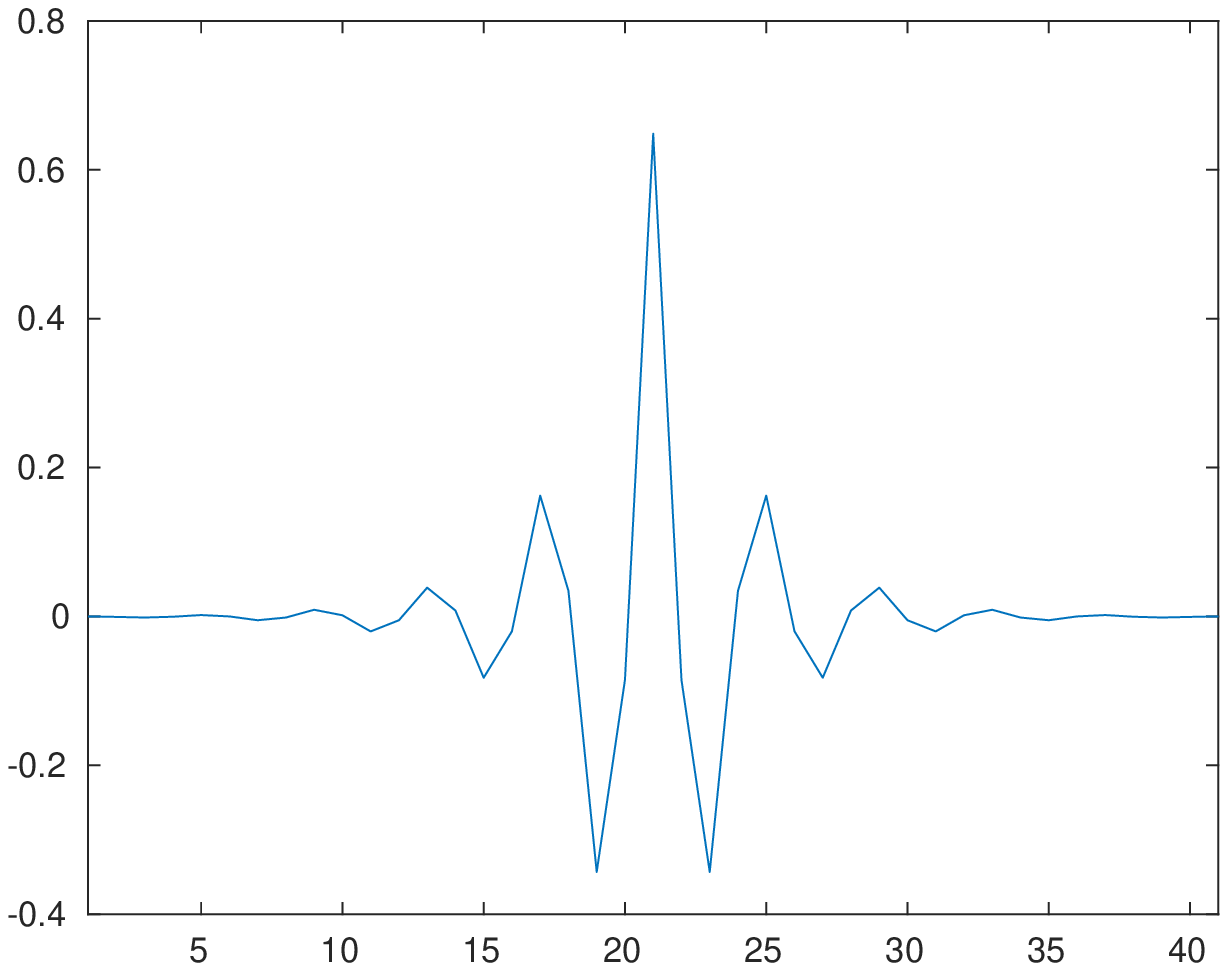}
    \caption{Edge Bloch mode illustration using (\ref{eq:finiteHill}) with $N=41$.}
    \label{fig:SSH13_edge}
\end{subfigure}
\caption{Edge mode in a SSH chain of unmodulated subwavelength resonators.}
\label{fig:SSH}
\end{figure}
\end{example}
\begin{rmk}
Using the above formulation, we can investigate the stability of the edge Bloch mode under time-modulation. To that end, we plot the Bloch modes for different amplitudes $\varepsilon$ of the modulation. Here we set the modulation frequency $\Omega =0.2$ and the number of resonators $N=41$. 
The phase shifts, defined in (\ref{modulation_parameters}), are given by $\phi_1=2\pi/3,\phi_2=4\pi/3,\phi_3 = 0$ for the first trimer on the left and then repeated periodically. Figure \ref{fig:SSH13} shows that the edge mode is stable with respect to time-modulation while the other (non-localized) modes  are sensitive to time-modulations as illustrated in Figure \ref{fig:SSH13_edge2}. 

\begin{figure}[H]
\centering
    \begin{subfigure}[t]{0.45\textwidth}
        \includegraphics[scale=0.5]{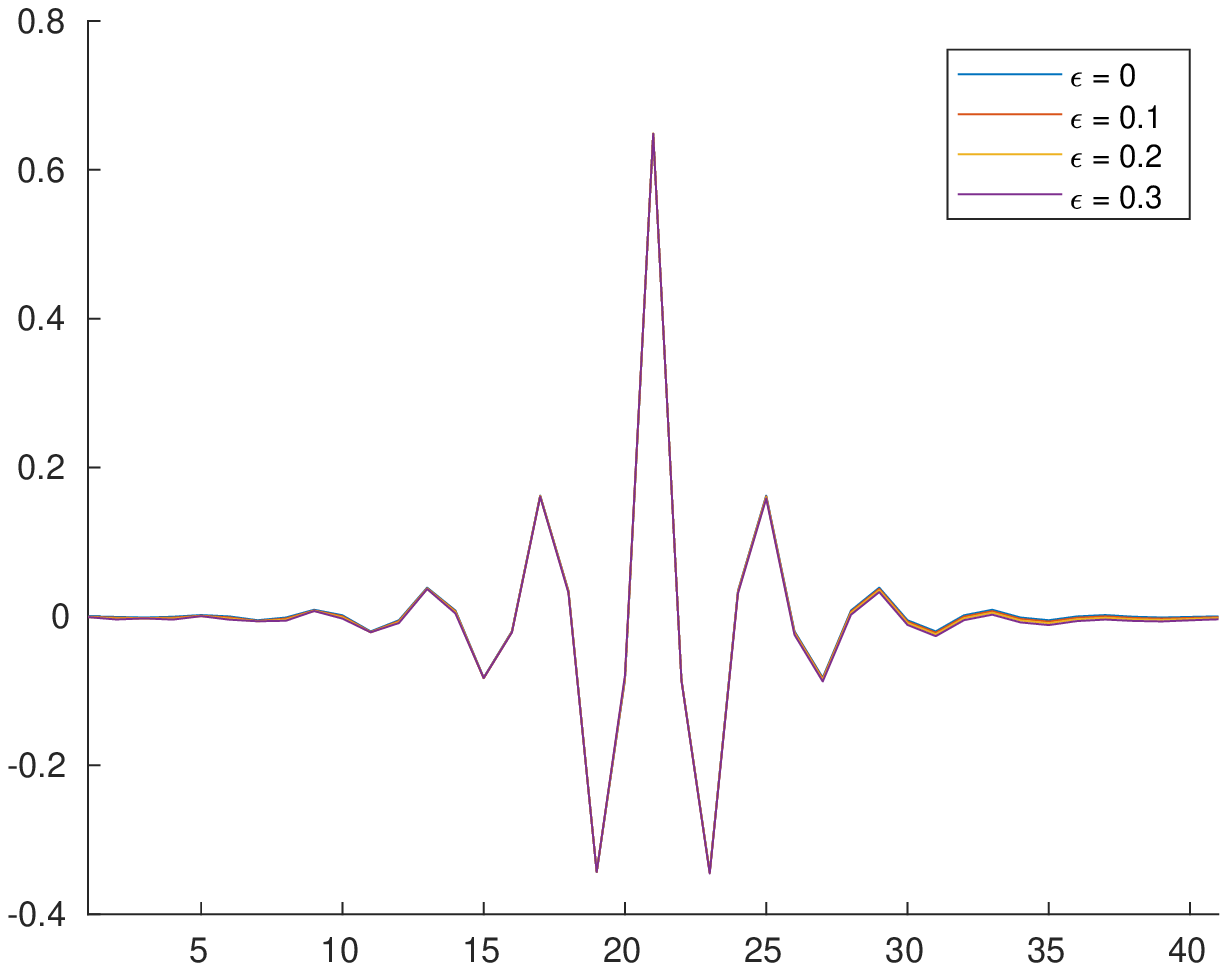}
        \caption{Edge Bloch modes under time-modulation for different modulation amplitudes.}
        \label{fig:SSH13_structure2}
    \end{subfigure}
\begin{subfigure}[t]{0.45\textwidth}
    \includegraphics[scale=0.5]{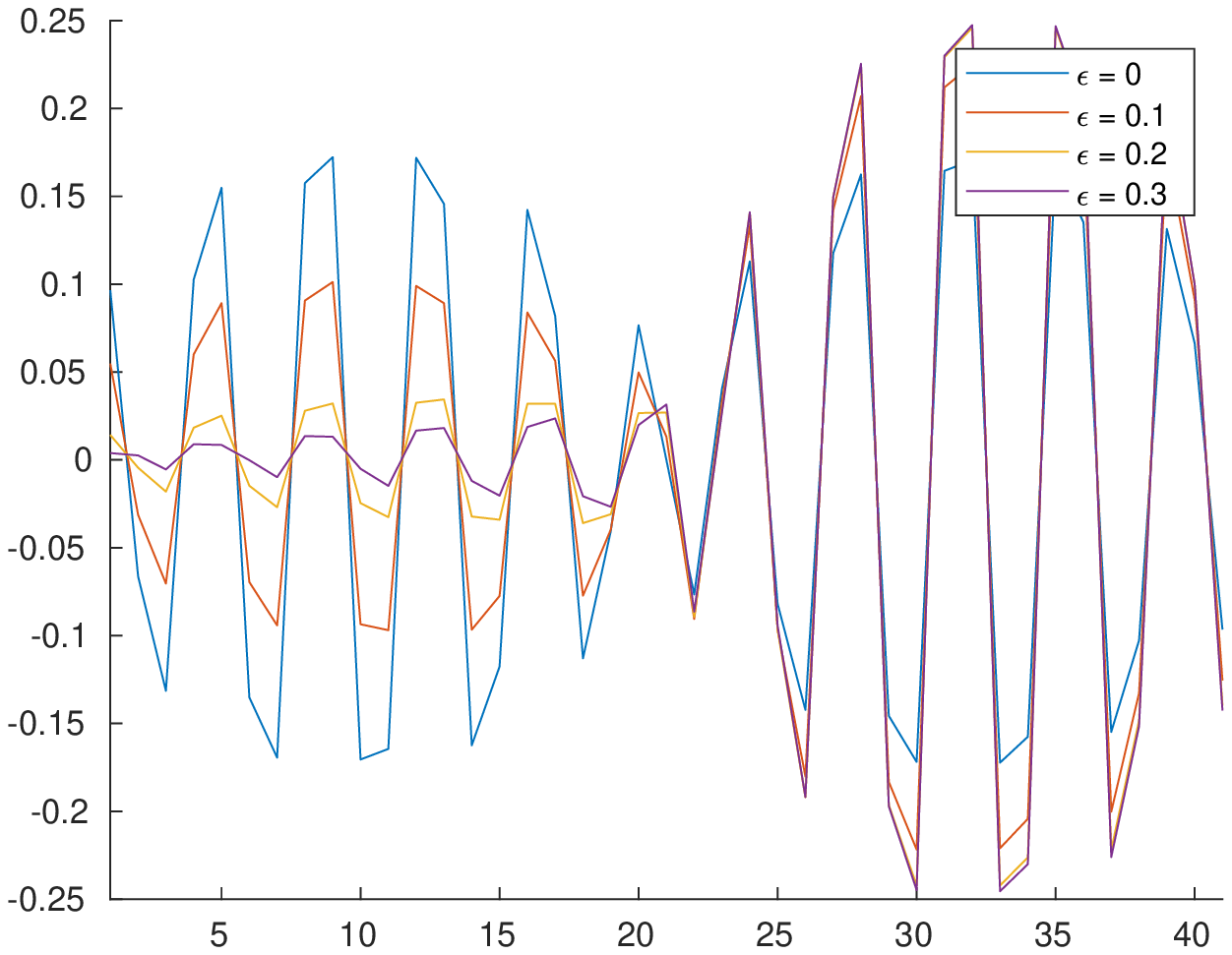}
    \caption{A certain non-edge Bloch mode under time-modulation for different modulation amplitudes.}
    \label{fig:SSH13_edge2}
\end{subfigure}
\caption{Stability of the edge mode under time-modulation.}  \label{fig:SSH13}
\end{figure}
\end{rmk}

With approximation (\ref{eq:approxiamtionbloch}) at hand, we also would like to derive some useful properties in the unmodulated case. First, in order to apply the general Floquet theory as outlined in Section \ref{sec:Floquettheory},  we rewrite the second order Hill's equation (\ref{Hill}) as a first order one by setting
\begin{equation*}
    y(t)=
    \begin{pmatrix}
    \phi(t) \\
    \frac{\mathrm{d}\phi}{\mathrm{d}t}(t)
    \end{pmatrix}.
\end{equation*}
We then obtain that
\begin{equation}
\label{firstorderHill}
A(t)y(t)=\frac{\mathrm{dy}}{\mathrm{d}t}(t),
\end{equation}
where
\begin{equation*}
    A(t):=\begin{pmatrix}
    0 & \mathrm{Id}_N\\
    -M(t) & 0
    \end{pmatrix}.
\end{equation*}

The following lemmas hold. 
\begin{lem}
Let $X(t)$ be the fundamental matrix of the Hill equation (\ref{firstorderHill}). Then $y^\omega(0)$ is an eigenvector of $X(T)$ associated with the eigenvalue $e^{\mathrm{i}\omega T}$.
\end{lem}
\begin{proof}
For each quasifrequency $\omega$, there exists $y^\omega(t)$ such that
\begin{equation}
\label{firstorderHill2}
    y^\omega(t+T) = e^{\mathrm{i}\omega T}y^\omega(t).
\end{equation}
Then we immediately have 
\begin{equation*}
    X(T)y^\omega(0)= y^\omega(T)=e^{\mathrm{i}\omega T}y^\omega(0).
\end{equation*}
\end{proof}
\begin{lem}
In the unmodulated case, the eigenvectors of $X(T)$ are given by those of $A$. That is, $v$ is an eigenvector of $A$ associated with the eigenvalue $\lambda$ if and only if $v$ is an eigenvector of $X(T)$ associated with the eigenvalue $e^{\lambda T}$. 
\end{lem}
\begin{proof}
In the unmodulated case, the fundamental matrix is given by
$\ds X(t)= e^{At}.$
The rest follows from Lie theory. 
\end{proof}
\begin{rmk}
\label{rmk:freq}
In the unmodulated case, $A$ is given by
\begin{equation*}
    A_{\varepsilon=0} = 
    \begin{pmatrix}
    0 & \mathrm{Id}\\
    \frac{\delta \kappa_r}{\rho_r}C & 0
    \end{pmatrix},
\end{equation*}
where $C$ is the capacitance matrix of the system, which is proven to be symmetric positive definite \cite{ammari2021functional}. The eigenvectors of $A$ come in pairs. In fact, $(v\ \pm av)^\top$ are eigenvectors of $A$, where $v$ is an eigenvector of $({\delta \kappa_r}/{\rho_r}) C$ with eigenvalue $a^2$ and $\top$ denotes the transpose. This is because
\begin{equation*}
    A\begin{pmatrix}
    v \\ \pm av
    \end{pmatrix} = \begin{pmatrix}
    \pm av \\ a^2v 
    \end{pmatrix}
    =\pm a \begin{pmatrix}
    v \\ \pm av
    \end{pmatrix} . 
\end{equation*}    
\end{rmk}

\section{Localization of edge modes in supercell structures}
\subsection{Degenerate edge modes in supercell structures}
In this subsection, we would like to introduce the so-called supercell structure. Defined similarly as in \cite{fleury2016floquet}, a supercell consists of two trimers time-modulated with two opposite orientations. Here, we consider all six disk resonators having the same radius $R$. A structure with $L$ supercells is defined as $L$ such supercells lined up equidistantly.
\begin{figure}[h]
\centering
\includegraphics[scale=0.7]{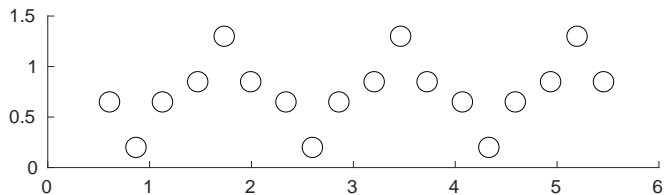}
\caption{Illustration of a supercell structure with three supercells (L=3).}
\label{fig:supercells}
\end{figure}
Mathematically, a supercell structure is defined as a system of disk resonators centered at
\begin{equation}
	\begin{split}
	& c_1 = (\frac{\sqrt{3}}{2},\frac{1}{2})+3R(\mathrm{cos}(\frac{5\pi}{6}),\mathrm{sin}(\frac{5\pi}{6})),  \ \ 
	 c_2 = (\frac{\sqrt{3}}{2},\frac{1}{2})+3R(\mathrm{cos}(\frac{3\pi}{2}),\mathrm{sin}(\frac{3\pi}{2})),\\
& c_3 =  (\frac{\sqrt{3}}{2},\frac{1}{2})+3R(\mathrm{cos}(\frac{\pi}{6}),\mathrm{sin}(\frac{\pi}{6})), \ \ 
	 c_4 =  (\sqrt{3},1)+3R(\mathrm{cos}(\frac{7\pi}{6}),\mathrm{sin}(\frac{7\pi}{6})),\\
& c_5 = (\sqrt{3},1)+3R(\mathrm{cos}(\frac{\pi}{2}),\mathrm{sin}(\frac{\pi}{2})), \ \ 
 c_6 = (\sqrt{3},1)+3R(\mathrm{cos}(\frac{11\pi}{6}),\mathrm{sin}(\frac{11\pi}{6})).
	\end{split}
\end{equation}
A structure with $L$ supercells is defined analogously as a system of disk resonators centered at
\begin{equation}
\label{periodicstructure}
	c_{i+6j} = c_i + j(\sqrt{3},0),
\end{equation}
for $i=1,\ldots,6$ and $j= 1,\ldots,L-1$; see Figure \ref{fig:supercells}.\\

Using the formulation of Bloch modes derived in the previous sections, we first prove the existence of degenerate edge modes. Then we show numerically that under time-modulation, these edge modes split into two, one localized on the left end side of the structure and the other on the right end side. To do so, we first determine algebraically the capacitance matrix corresponding to the $L$-supercell structure based on the symmetry of the structure. As seen in Remark \ref{rmk:freq}, the Bloch modes in the time-independent case are closely related to the eigenvectors of the capacitance matrix.
\begin{thm}
\cite{feppon}
\label{thm:symmetries}
(Symmetry of the capacitance matrix) Let $D\subset \mathbb{R}^3$ be  a system of $N$ resonators. If there exists a permutation $\sigma$ of the resonators  associated with the isometry $S$: $\mathbb{R}^3\rightarrow  \mathbb{R}^3 $ such that $SD=D$, then the coefficients of the capacitance matrix satisfy
    \begin{equation}
        C_{\sigma(i),\sigma(j)} = C_{i,j}.
    \end{equation}
\end{thm}

\begin{figure}[H]
\centering
\includegraphics[scale=0.5]{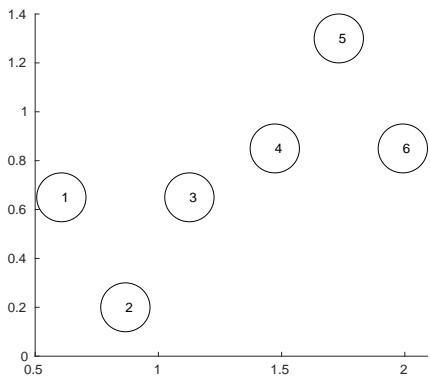}
\caption{Illustration of a structure with one supercell $(L=1)$, where the resonators are enumerated from left to right.} \label{fig:one_supercell_enumerated}
\end{figure}

\begin{example}
\label{thm:structureC}
For the case of one supercell, we can enumerate the resonators as in Figure \ref{fig:one_supercell_enumerated}. The geometric structure is invariant under the permutations
\begin{equation}
    \sigma_1=(1\ 2)(5\ 6),\ \sigma_2=(1\ 5)(2\ 6)(3\ 4)\ \text{and}\ \sigma_3 = (1\ 6)(2\ 5)(3\ 4).
\end{equation}
By Theorem \ref{thm:symmetries}, being invariant under $\sigma_1$ implies that
\begin{equation}
    C_{1,1} = C_{2,2}, C_{5,5}= C_{6,6}, C_{1,3}=C_{2,3}, C_{1,4}=C_{3,4}, C_{1,5}=C_{2,6}\ \textrm{and} \ C_{1,6}=C_{2,5}. 
\end{equation}
From this and the invariance  under $\sigma_2$ and $\sigma_3$ of the unit supercell structure (i.e., with $L=1$),  we conclude that the symmetric capacitance matrix reads
\begin{equation}
\label{eq:singlesupercellcap}
C=\begin{pmatrix}
a & c & k & d & b & h \\
c & a & k & d & h & b \\
k & k & e & f & d & d \\
d & d & f & e & k & k \\
b & h & d & k & a & c \\
h & b & d & k & c & a
\end{pmatrix},
\end{equation}
\end{example}
where $a, \ldots, k \in \mathbb{R}$. 

In \cite{ammari2020topological}, the asymptotic behaviour of the capacitance matrix and its matrix elements in the dilute regime is described by the following theorem. 

\begin{thm}
\label{thm:dilute}
 For a given resonators system with equal sizes in the dilute regime
    \begin{equation}
        D = \bigcup_{j=1}^N(B+\gamma^{-1}z_j),
    \end{equation}
    the entries of the capacitance matrix have the following asymptotic behavior:
\begin{equation}
C_{i,j}= \begin{cases}\operatorname{Cap}_{B}+O\left(\gamma^{2}\right), & i=j, \\ -\frac{\gamma \operatorname{Cap}_B^2 }{4\pi \left|z_{i}-z_{j}\right|}+O\left(\gamma^{2}\right), & i \neq j, \end{cases}\end{equation}
where the scaling parameter $\gamma\rightarrow 0$ and $\gamma^{-1}z_j$ represents the position of each resonator. Here, the capacity $\mathrm{Cap}_B$ of $B$ is given by
\begin{equation}
	\mathrm{Cap}_B=-\int_{\partial B}(\mathcal{S}_B)^{-1}[\chi_{\partial D}]\mathrm{d}\sigma
\end{equation}
with $\mathcal{S}_B$ being the single-layer potential associated with the Laplacian on $\partial B$. 
\end{thm}

Note that $z_j= \gamma^{-1} c_j, j=1,\ldots, N$, where the $c_j$'s are defined in (\ref{periodicstructure}). 

The next result  is our first main result in this paper.  It shows that the approximate degeneracy of at least one of the eigenvalues of the capacitance matrix $C$. 
\begin{thm}
\label{thm:main_degenerate}
    Let $L$ denote the number of supercells and $C$ the corresponding capacitance matrix. There exists in the dilute regime, up to an error of the order of $\gamma d$, an eigenvalue $\lambda$ with multiplicity two. Here, $d$ is a small constant fixed by the distance between the centers of the neighbouring resonators. Moreover, the corresponding eigenvectors are given by
    \begin{equation}\label{v1v2}
        v_1 = 
        \begin{pmatrix}
        -1\\
        1\\
        0_{N-4}\\
        1\\
        -1
        \end{pmatrix}        
    \textrm{and}\ 
       v_2 = 
        \begin{pmatrix}
        -1\\ 1\\ 0_{N-4}\\ -1\\ 1\\
        \end{pmatrix},
    \end{equation}
and we have
\begin{equation*}
	(C - \lambda\mathrm{Id}_N) v_i = v^i_\delta,\ \ i=1,2, 
\end{equation*}
where the entries of the error term $v^i_\delta$ are of order $O(\gamma  d)$. 
\end{thm}
\begin{proof}
Fix $L$ and let $N=6L$. Then, enumerate the resonators from left to right as in Figure \ref{fig:one_supercell_enumerated}. Let $C$ denote the associated capacitance matrix. We note that the  system of resonators has the  $\pi$-rotational isometry associated with the permutation:
\begin{equation*}
    \sigma = \prod_{j=1}^{\frac{N}{2}}(j,\ N+1-j)= (1,\ N)\cdot(2,\  N-1) \cdots (\frac{N}{2},\ \frac{N}{2}+1).
\end{equation*}
By Theorem \ref{thm:symmetries}, we obtain that
\begin{equation}
\label{eq:permutation}
    C_{i,j}=C_{N+1-i,N+1-j}.
\end{equation}
Let $w:=Cv_1$. Then the $j$th entry of $w$ is given by
\begin{equation}
    w_j = -C_{j,1}+C_{j,2}+C_{j,N-1}-C_{j,N}.
\end{equation}
Define $\lambda:=w_2$ and compute for $j=1$:
\begin{equation}
\label{eq:estimate12}
\begin{split}
    -w_1-\lambda &= -(-C_{1,1}+C_{1,2}+C_{1,N-1}-C_{1,N})-(-C_{2,1}+C_{2,2}+C_{2,N-1}-C_{2,N})\\
    &= (C_{1,1}-C_{2,2})+(C_{1,N}-C_{2,N-1})\\
    & = C_{1,N}-C_{2,N}+C_{2,N}-C_{2,N-1}+O(\gamma^2)\\
    & = O(\gamma d L^{-2}),
\end{split}
\end{equation}
where $d = \lvert z_1-z_2\rvert$.
In the first step, we have used the symmetry of $C$ and (\ref{eq:permutation}) to conclude that $C_{1,2}=C_{2,1}$ and $C_{1,N-1}=C_{N,2}=C_{2,N}$. In the third step, we have estimated the quantity $C_{i,j}- C_{i,k}$ as follows:
\begin{equation}
\label{eq:estimateL}
    \begin{split}
        -{4\pi}(C_{i,j}-C_{i,k}) & = \gamma (\frac{1}{\lvert z_i-z_j\rvert}-\frac{1}{\lvert z_i-z_k\rvert})+O(\gamma^2)\\
        & \leq {\gamma}(\frac{1}{\lvert z_i-z_j\rvert}-\frac{1}{\lvert z_i-z_j\rvert+\lvert z_j-z_k\rvert})+O(\gamma^2)\\
        & = \gamma(\frac{\lvert z_j-z_k\rvert }{\lvert z_i-z_j\rvert ^2}-\frac{\lvert z_j-z_k\rvert^2 }{\lvert z_i-z_j\rvert ^3}+\ldots)+O(\gamma^2).
    \end{split}
\end{equation}
Setting $i=N,j=1,k=2$ or $i=2,j=N,k=N-1$ in the above estimate, we arrive at
\begin{equation}
    C_{1,N}-C_{2,N}=O(\gamma d L^{-2})\ \text{and}\ C_{2,N}-C_{2,N-1}=O(\gamma d L^{-2}). 
\end{equation}
Similarly, we can estimate the entries $k=3,\ldots,N-3$:
\begin{equation*}
    w_k = -C_{k,1}+C_{k,2}+C_{k,N-1}-C_{l,N}
    = O(\gamma d).
\end{equation*}
We remark that the above entries are often zero by symmetry. Finally, the estimations of $w_{N-1}$ and $w_N$ can be derived in the same way as in (\ref{eq:estimate12}) to be $\pm \lambda + O(\gamma d)$. Hence, we conclude that
\begin{equation*}
    Cv_1 = \lambda v_1 +v_\delta^1,
\end{equation*}
where the entries of $v_\delta^1$ are of the order of $\gamma d$. The same result  holds for $v_2$.
\end{proof}
\begin{rmk}
    If the number of the supercells $L$ is large enough, then the entries $w_{1}$ and $w_{2}$ can be estimated to be $\pm \lambda+O(\gamma^2)$ as in  (\ref{eq:estimate12}). The same holds for the last two entries $w_{N-1}$ and $w_{N}$. 
\end{rmk}

\begin{example}
Taking $L=1$, i.e., considering the unit supercell structure, the capacitance matrix is given by (\ref{eq:singlesupercellcap}). Hence, we can verify that
\begin{equation*}
    Cv_1=
    \begin{pmatrix}
    -a+c+b-h\\ a-c-b+h\\ 0\\ 0\\ a-c-b+h \\ -a+c+b-h 
    \end{pmatrix}=(a-c-b+h)v_1
\end{equation*}
and 
\begin{equation*}
    Cv_2=
    \begin{pmatrix}
    -a+c-b+h\\ a-c+b-h\\ 0\\ 0\\ -a+c-b+h \\ a-c+b-h 
    \end{pmatrix}=(a-c+b-h)v_2.
\end{equation*}
Define 
\begin{equation*} \label{deflambda2}
\lambda:=a-c-b+h \quad \mbox{and} \quad \lambda^\prime:=a-c+b-h, \end{equation*} and estimate
\begin{equation*}
    \eta = |\lambda -\lambda^\prime| = 2\lvert b-h \rvert=O(\gamma  \lvert z_1-z_2\rvert).
\end{equation*}
Hence, we obtain that
\begin{equation*}
	Cv_2 -\lambda v_2 = \eta v_2.
\end{equation*}
This verifies Theorem \ref{thm:main_degenerate} for $L=1$.
\end{example}

\begin{example}
\label{exp:statics}
Here, we provide an example where $L=4$ and illustrate our results numerically in Figures \ref{fig:supercell_edge1}  and \ref{fig:supercell_edge2}. We call edge
 modes the modes localized at the left and right edges of the supercell structure. In Figure \ref{figexpm}, using the multipole expansion method (see, for instance, \cite[Appendix C]{ammari2017subwavelength}), we show the edge modes corresponding to the vectors $v_1$ and $v_2$ defined in (\ref{v1v2}). Because of the symmetry of the unmodulated structure, these edge modes are localized both on the right and  the left of the structure. 

\begin{figure}[H]
\begin{subfigure}[b]{0.45\textwidth}
\centering
    \includegraphics[width=\textwidth]{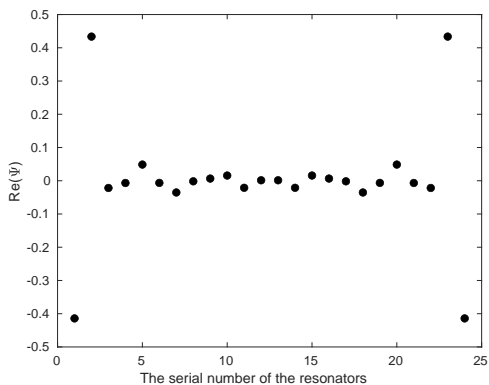}
    \caption{The first edge mode corresponding to $v_1$ as in Theorem \ref{thm:main_degenerate}.}
    \label{fig:supercell_edge1}    
\end{subfigure}
     \hfill
\begin{subfigure}[b]{0.45\textwidth}
\centering
    \includegraphics[width=\textwidth]{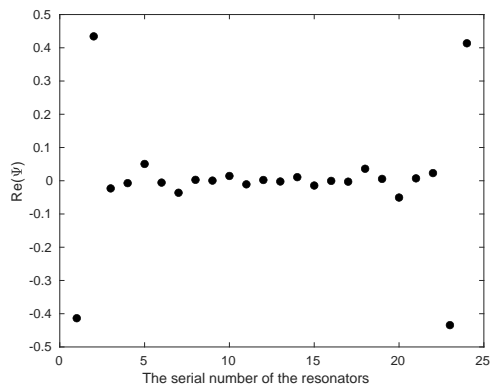}
    \caption{The second edge mode corresponding to $v_2$ as in Theorem \ref{thm:main_degenerate}.}
    \label{fig:supercell_edge2}    
\end{subfigure}
     \hfill
    \caption{Edge mode simulations using the multipole expansion method verifying the results proved in Theorem \ref{thm:main_degenerate}.}
    \label{figexpm}
\end{figure}
\end{example}
\begin{rmk}
As seen in Remark \ref{rmk:freq}, if $v$ is an eigenvector of the capacitance matrix $C$ associated with the eigenvalue $\lambda$,  then $\begin{pmatrix}
v \\ {\mu} v
\end{pmatrix}$ is an eigenvector of $A_{\varepsilon}$ associated with the eigenvalue $${\mu}=\sqrt{-\frac{\delta \kappa_r}{\rho_r}\lambda}.$$ Observe that the eigenvectors $v_1$ and $v_2$ of the capacitance matrix of the supercell structures corresponding to the approximately degenerate eigenvalue $\lambda$.  They take their maximum or minimum values on the boundary of the structure. Hence, we refer to $\mu$ as the edge degenerate subwavelength quasifrequency and to $v_{1}$ and $v_2$ as the corresponding edge modes.  
\end{rmk}

\begin{rmk}
Now,  we turn to Example \ref{example_topo}. We first plot the bandgap structure of a supercell repeated periodically along the $x$-axis. Then we set 
$\Omega=0.2$, $\varepsilon=0.2$, and the phase shifts for the unit supercell to be given by $\phi_1=2\pi/3,\phi_2=4\pi/3,\phi_3 = 0,\phi_4 = 2\pi/3,\phi_5 =0, \phi_6=4\pi/3$. Note that in the unit supercell the first and second trimers are modulated with  opposite orientations. Figure  \ref{fig:supercell_dirac} shows the Dirac degeneracies (at the symmetry points $\pm \pi/L$) and their non-reciprocal openings due to time-modulations of this infinite periodic supercell structure. 
This non-reciprocal bandgap opening by time-modulating the supercell structure is due to broken time-reversal symmetry; see \cite{paper1}.

\begin{figure}[H]
\begin{subfigure}[b]{0.45\textwidth}
\centering
    \includegraphics[width=\textwidth]{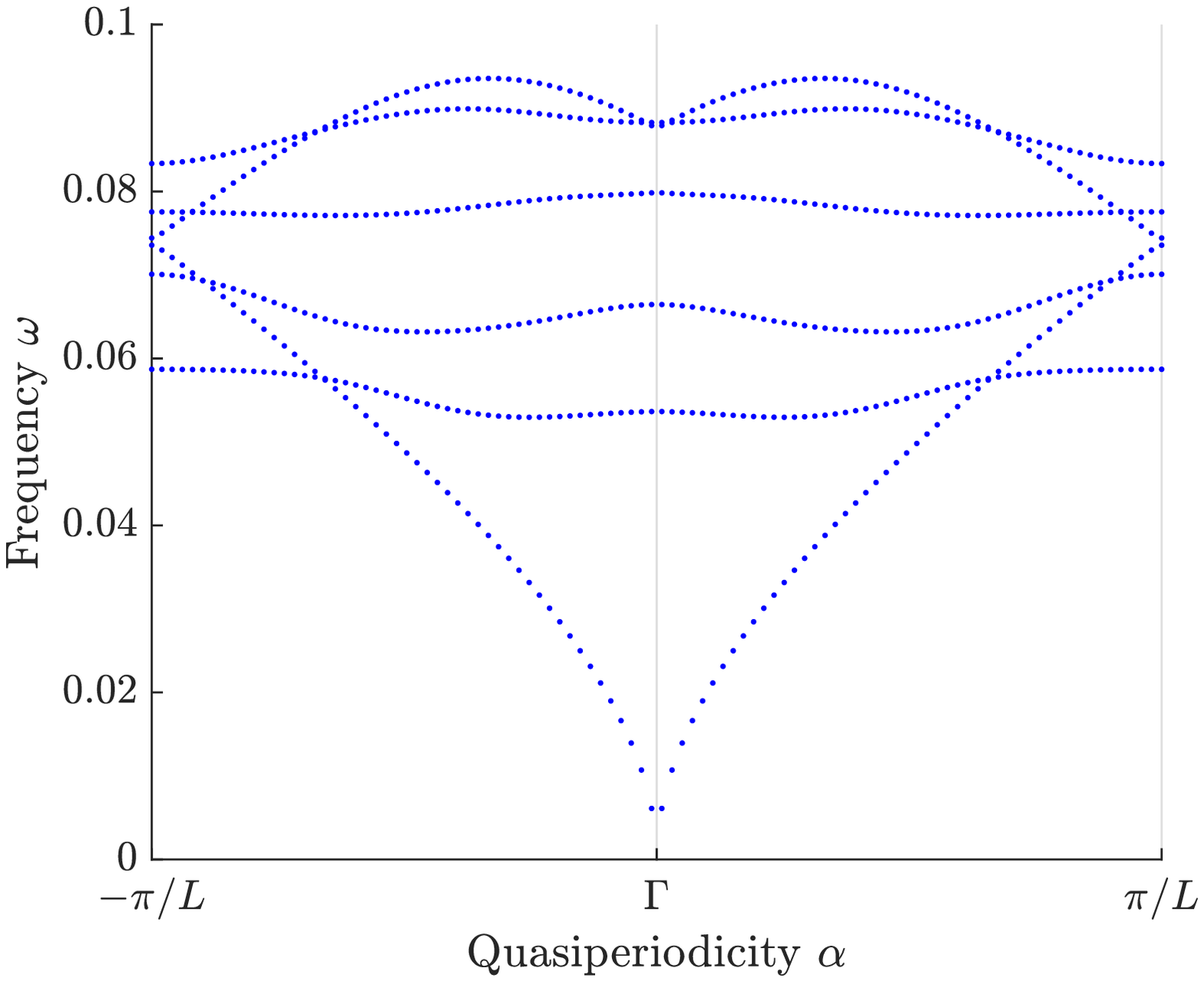}
    \caption{Dirac degeneracies for the unmodulated supercell structure.}
    \label{fig:supercell_dirac_structure}    
\end{subfigure}
     \hfill
\begin{subfigure}[b]{0.45\textwidth}
\centering
  \includegraphics[width=\textwidth]{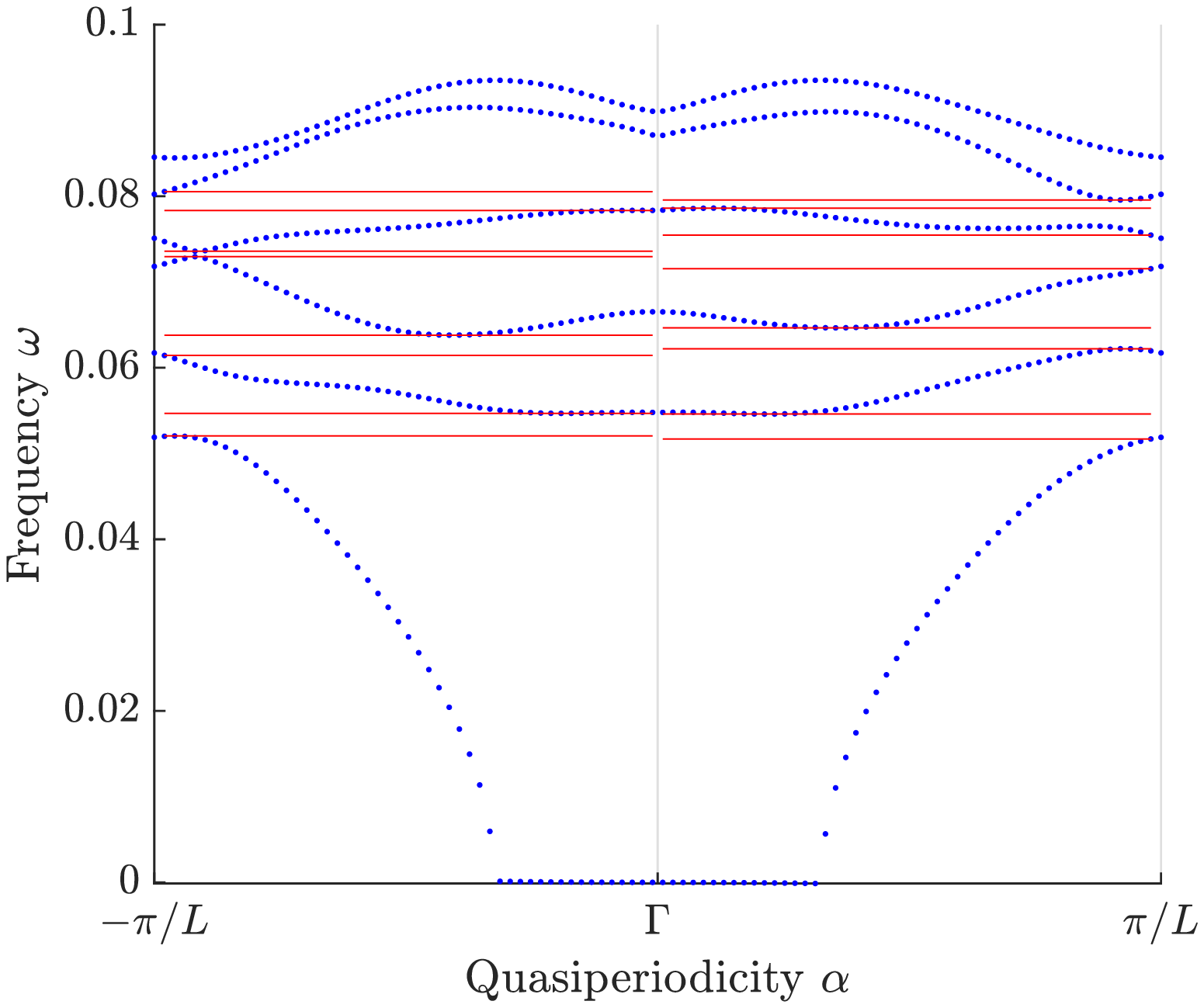}
    \caption{Non-reciprocal bandgap opening by time-modulating the supercell structure.}
    \label{fig:supercell_dirac_open}    
\end{subfigure}
     \hfill
    \caption{Dirac degeneracies and their non-reciprocal openings due to time-modulations of an infinite periodic supercell structure.} \label{fig:supercell_dirac}
\end{figure}
\end{rmk}
\begin{rmk}
In Figures \ref{fig:supercells_unmod} and \ref{fig:supercells_mod} we make an interface between two supercell structures by applying a mirror symmetry. In both the unmodulated and modulated cases, we can see that they are three edge modes, one in the middle and two on the side edges. By turning on the time-modulations,  all three edge modes are stable. In particular, the two side edge modes do not split.

\begin{figure}[H]
\centering
\includegraphics[scale=0.6]{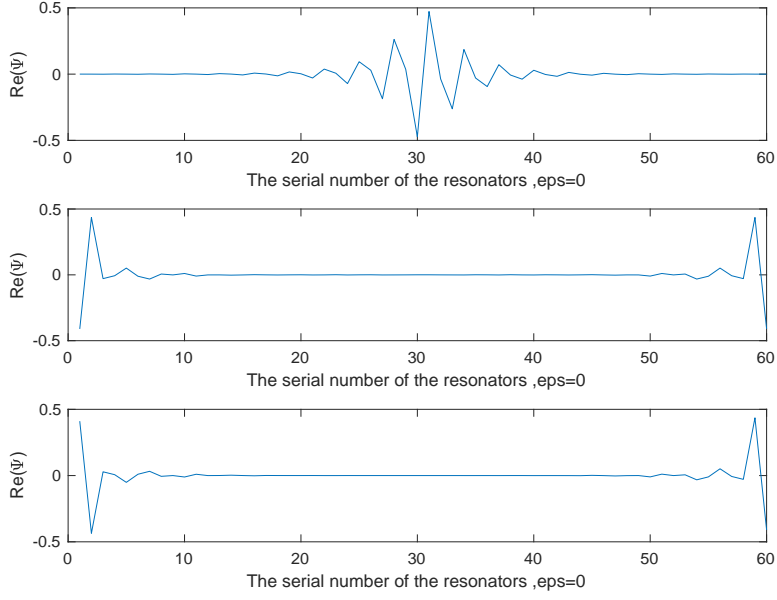}
\caption{Edge modes in mirror symmetric supercell structure in the unmodulated case.}
\label{fig:supercells_unmod}
\end{figure}

\begin{figure}[H]
\centering
\includegraphics[scale=0.6]{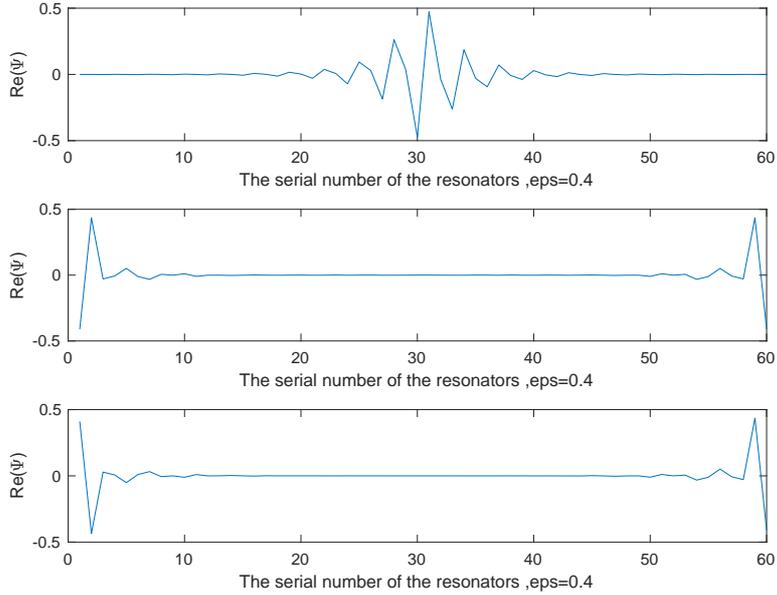}
\caption{Edge modes in mirror symmetric supercell structure in the modulated case with $\Omega=0.2$ and $\varepsilon=0.4$.}
\label{fig:supercells_mod}
\end{figure}

\end{rmk}

\subsection{Localization of edge modes in the time-modulated case}
In this subsection, we prove that the degenerate edge eigenspace, under time modulation, splits into two one-dimensional eigenspaces described by an analytical formula. We numerically verify that these one-dimensional eigenspaces  correspond to the left and right edge modes.
By turning on the time-modulation with amplitude $\varepsilon$, the first order Hill's equation (\ref{firstorderHill}) becomes a Hill equation with time-dependent coefficients:
\begin{equation} \label{4.25}
\frac{\mathrm{d}y(t)}{\mathrm{d}t} = A_\varepsilon(t)y(t).
\end{equation} 
Applying Floquet theorem to the ODE (\ref{4.25}), its fundamental solution can be written as
\begin{equation*}
X(t) = P_\varepsilon(t)\mathrm{exp}(F_\varepsilon t).
\end{equation*}
Here, $F_\varepsilon$ can be expanded in terms of $\varepsilon$  as follows:
\begin{equation*}
F_\varepsilon = F_0 + \varepsilon F_1 +\varepsilon^2 F_2 +O(\varepsilon^3).
\end{equation*}
We have proved in the last section that, in the dilute regime, $F_0$ has a degenerate eigenvalue ${\nu_0}$ with multiplicity $2$. Denote the two-dimensional eigenspace by $E_{\nu_0}=\mathrm{span}\{v_1,v_2\}$. By turning on the-modulation, $E_{\nu_0}$  splits into two one-dimensional eigenspaces corresponding to the eigenvalues:
\begin{equation*}
\begin{split}
&\nu_w = \nu_0 + \varepsilon \nu^w_1 + \varepsilon^2 \nu^w_2 +O(\varepsilon^3),\\
&\nu_q = \nu_0 + \varepsilon \nu^q_1 + \varepsilon^2 \nu^q_2 +O(\varepsilon^3),
\end{split}
\end{equation*}
which are spanned by the eigenvectors $w$ and $q$ given by
\begin{equation*}
\begin{split}
&w = w_0  + \varepsilon w_1 + \varepsilon^2 w_2 +O(\varepsilon^3),\\
&q = q_0  + \varepsilon q_1 + \varepsilon^2 q_2 +O(\varepsilon^3),
\end{split}
\end{equation*}
where 
\begin{equation} \label{4.30} 
w_0=\beta^w_1 v_1+\beta^w_2 v_2\in E_{\nu_0} \quad \mbox{and} \quad  q_0=\beta^q_1 v_1+\beta^q_2 v_2\in E_{\nu_0}. \end{equation}
Our second main result is stated in the following  theorem. 
\begin{thm}
    The perturbation due to time-modulation of the edge degenerate eigenvalue $\nu$  is at least of order two in the modulation amplitude $\varepsilon$, i.e., 
    \begin{equation*}
    \nu_1^w = \nu_1^q = 0.
    \end{equation*}
\end{thm}

\begin{proof}
Theorem \ref{thm:main_degenerate} provides an (asymptotically) degenerate eigenvalue $\lambda$ of the capacitance matrix. By Remark \ref{rmk:freq}, the quasifrequencies  in the static case for the ODE (\ref{ode}) are given by the square roots of the eigenvalues of the capacitance matrix \cite{ammari2021functional}. They have the same folding number in terms of Definition \ref{foldingnumber}. Hence, by Theorem \ref{thm:expansionFloquet} and the eigenvalue perturbation theory (see Appendix \ref{appendixA}), the first-order term $PF_1P$ is a $2\times 2$ zero matrix. This is indeed because $A_1$ does not have constant Fourier coefficients as derived in (\ref{eq:expansionMij}). Thus we conclude that the first-order perturbation is zero and therefore the perturbation due to time-modulation of the edge degenerate eigenvalue $\nu$  is at least of order two.
\end{proof}

The coefficients in (\ref{4.30}) are determined as follows. 
\begin{thm}
\label{thm:localisation}
Denote $P$ the projection operation onto the eigenspace $E_{\nu_0}$ and let $B$ be defined by \begin{equation*}
    B= (PF_1GF_1P+PF_2P)|_{E_{\nu_0}},
\end{equation*}
and $G:=(\nu \, \text{Id} -F_0)^{-1}(1-P)$. Then the following matrix equations hold:
\begin{equation*}
\begin{split}
B \begin{pmatrix}
\beta^w _1 \\ \beta^w _2
\end{pmatrix} & = \nu_2^w \begin{pmatrix}
\beta^w _1 \\ \beta^w _2
\end{pmatrix},\\
B \begin{pmatrix}
\beta^q _1 \\ \beta^q _2
\end{pmatrix} & = \nu_2^q \begin{pmatrix}
\beta^q _1 \\ \beta^q _2
\end{pmatrix}.\\
\end{split}
\end{equation*}
\end{thm}
\begin{proof}
The proof is given in Appendix \ref{appendixA}.
\end{proof}
\begin{example}
Here, we continue investigating Example \ref{exp:statics} where $L=4$.
 We implement an algorithm to compute the Floquet matrix elements as in Theorem \ref{thm:expansionFloquet} and plug them into the analytical formula given in Theorem \ref{thm:localisation} to determine $w_0$ and $q_0$. Furthermore, we compute the first-order perturbation in the eigenvectors $w_1$ and $q_1$ using equations outlined in (\ref{eq:firstordervector}). Setting $w_{\text{red}}=w_0+\varepsilon w_1$, we plot it in Figure \ref{fig:supercell_edge11} against the Bloch mode $w_{\text{black}}$ obtained by using the multipole expansion method. Moreover,  we do the same for $q$ in Figure \ref{fig:supercell_edge21}. The modulation amplitude is set to be $\varepsilon=0.2$, the modulation frequency $\Omega = 2$ and the phase shifts  as those in Figure \ref{fig:supercell_dirac}. As defined in (\ref{modulation_parameters}), they are given by $\phi_1=2\pi/3,\phi_2=4\pi/3,\phi_3 = 0,\phi_4 = 2\pi/3,\phi_5 =0, \phi_6=4\pi/3$ for the first cell and then repeated periodically as in (\ref{periodicstructure}). We emphasize that the error term is due to higher-order effects in $\varepsilon$.
\begin{figure}[H]
\begin{subfigure}[b]{0.45\textwidth}
\centering
    \includegraphics[width=\textwidth]{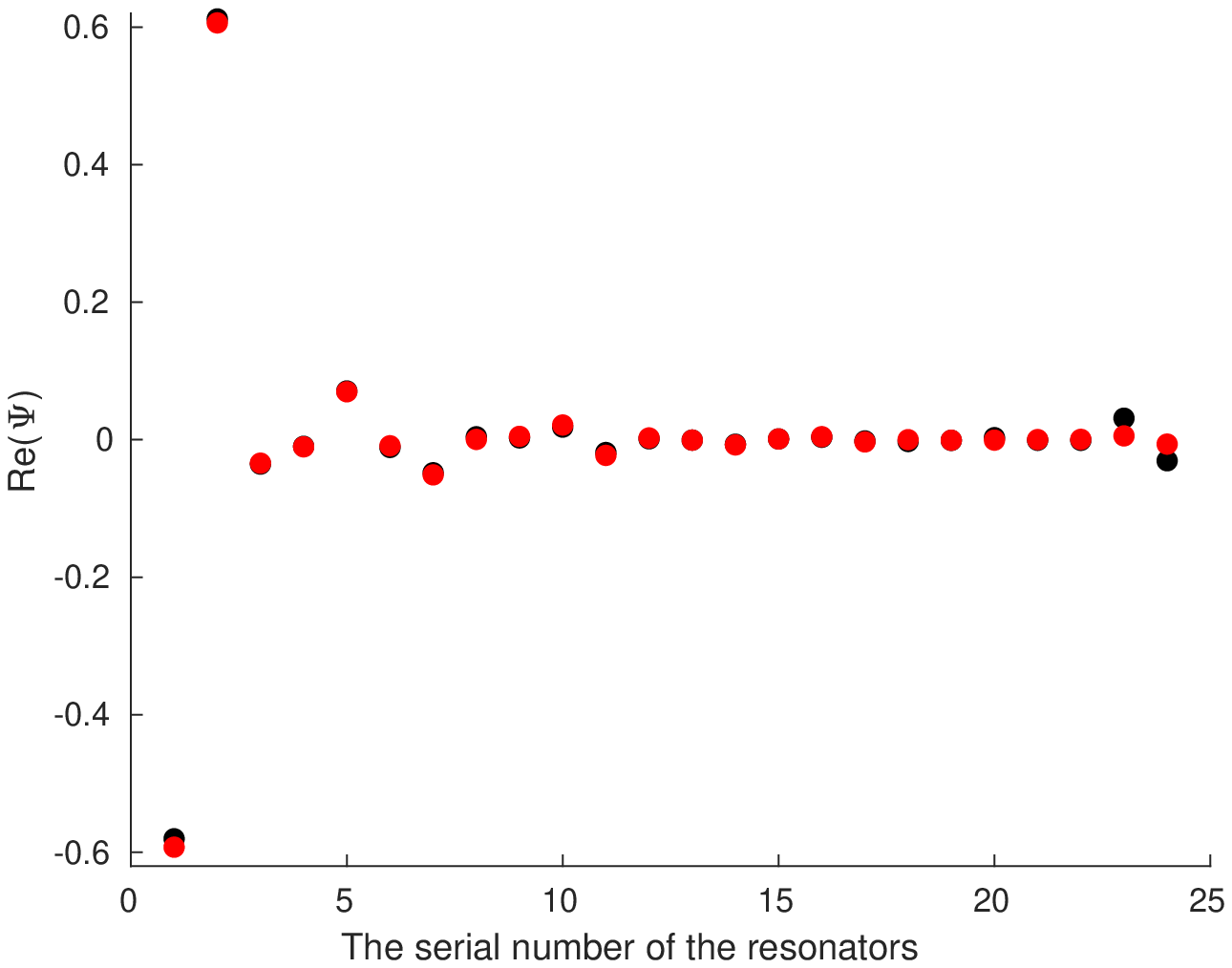}
    \caption{The first edge mode comparison.}
    \label{fig:supercell_edge11}    
\end{subfigure}
     \hfill
\begin{subfigure}[b]{0.45\textwidth}
\centering
    \includegraphics[width=\textwidth]{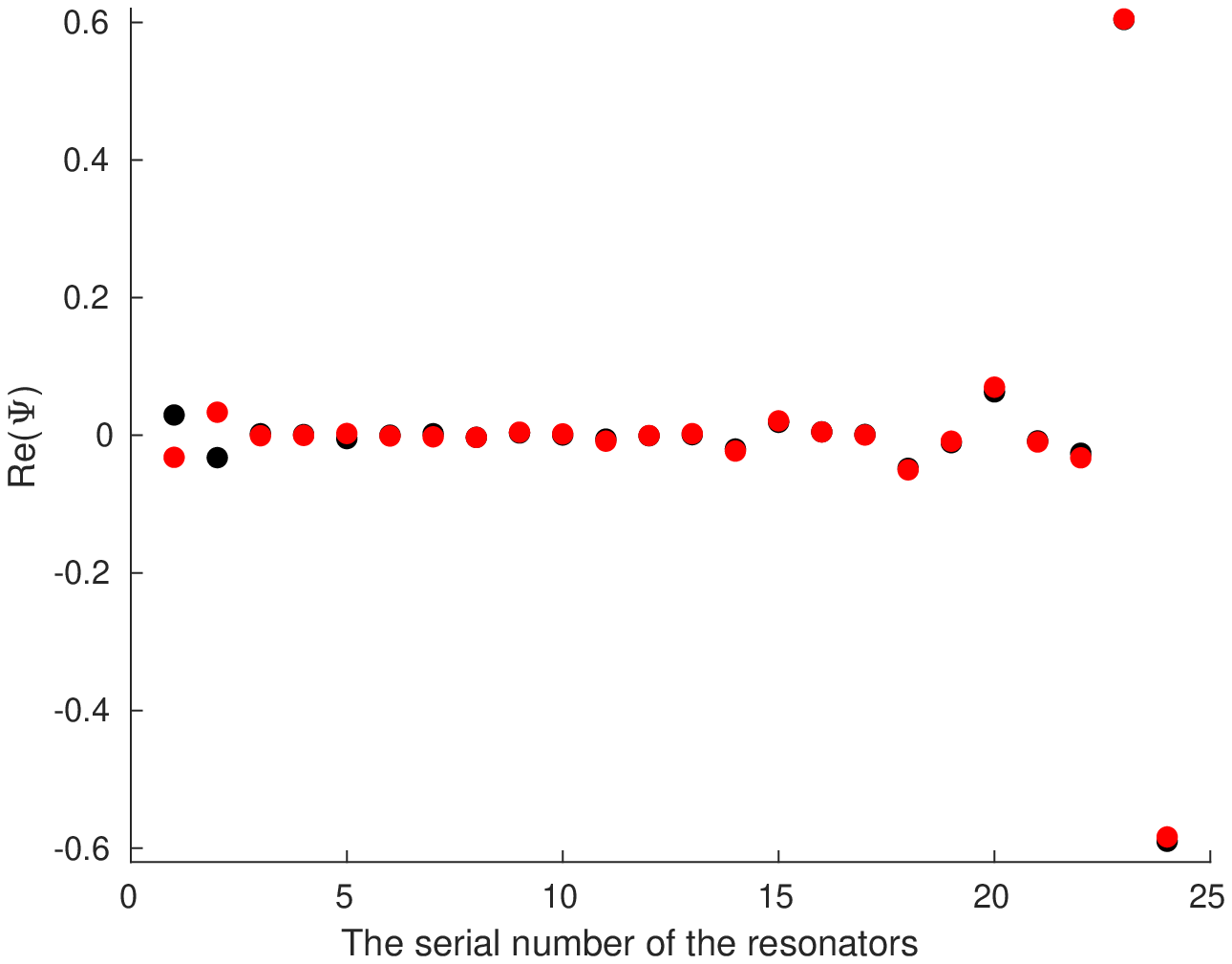}
    \caption{The second edge mode comparison.}
    \label{fig:supercell_edge21}    
\end{subfigure}
     \hfill
    \caption{Edge mode simulations using the multipole expansion method (in black) compared with those predicted asymptotically by Theorem \ref{thm:localisation} (in red).}
\end{figure}
\end{example}

As shown in Figures \ref{fig:supercell_edge11}  and \ref{fig:supercell_edge21}, the modes in the unmodulated case computed in Figures \ref{fig:supercell_edge1} and \ref{fig:supercell_edge2} split under time-modulation into two unidirectional edge modes, each of them is localized at one of the two edges of the structure. 

\subsection{Numerical simulation of the robustness of the edge modes}
In this subsection, we discuss the robustness of the edge modes in the supercell structures with respect to the modulation amplitude $\varepsilon$. Using numerical simulations, we demonstrate stability of the edge modes in the case of large $N$ and small perturbations in the modulation amplitude $\varepsilon$.
\begin{example}
In this example, we set $N$ to be $120$, i.e., we consider a structure with  $20$ supercells. In Figure \ref{f7}, we illustrate the localization effect as described in the previous sections. 

\begin{figure}[H]
\begin{subfigure}[b]{0.45\textwidth}
\centering
    \includegraphics[width=\textwidth]{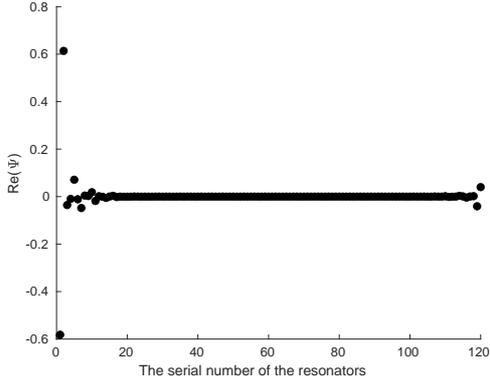}
    \caption{The first edge mode localized on the left.}
    \label{fig:supercell_edge11b}    
\end{subfigure}
     \hfill
\begin{subfigure}[b]{0.45\textwidth}
\centering
    \includegraphics[width=\textwidth]{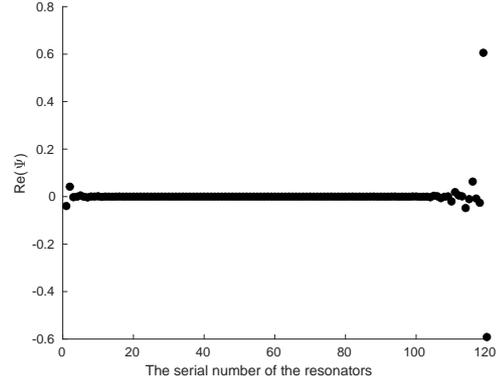}
    \caption{The second edge mode localized on the right.}
    \label{fig:supercell_edge21b}    
\end{subfigure}
     \hfill
    \caption{Edge mode simulations using the multipole expansion method with $N=120$.}
    \label{f7}
\end{figure}
\end{example}

\begin{example}
In this example, we investigate the robustness of the edge modes when adding an error term to the modulation amplitude $\varepsilon=0.2$. For each mean value $\mu$, we generate a normally distributed random array $e_\mu$ with mean $\mu$. We perform numerical simulations on the supercell structure with modulation amplitude $(0.2,\ldots,0.2)+e_\mu$ to obtain the edge modes. After repeating this $1000$ times, we take the average of all the $1000$ edge modes. We connect the values in the visualizations to demonstrate the differences. Figure \ref{f8} shows that the edge modes we observe are stable under small perturbations in the modulation amplitude.
\begin{figure}[H]
\begin{subfigure}[b]{0.45\textwidth}
\centering
    \includegraphics[width=\textwidth]{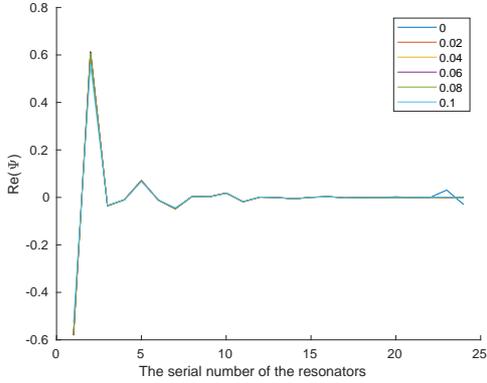}
    \caption{The first edge mode localized on the left.}
    \label{fig:supercell_edge11c}    
\end{subfigure}
     \hfill
\begin{subfigure}[b]{0.45\textwidth}
\centering
    \includegraphics[width=\textwidth]{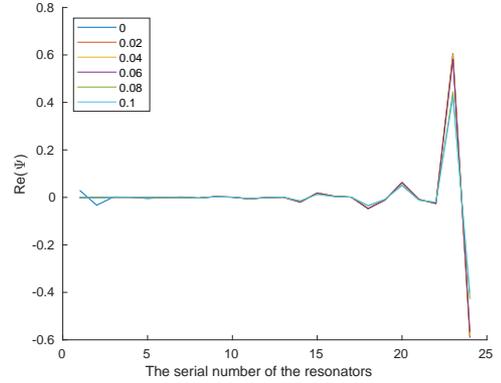}
    \caption{The second edge mode localized on the right.}
    \label{fig:supercell_edge21c}    
\end{subfigure}
     \hfill
    \caption{Edge mode simulations using the multipole expansion method with Gaussian error term in the modulation amplitude. The legend marks the different mean values of the error term.} \label{f8}
\end{figure}
\end{example}

\appendix 
\section{Eigenvalue and eigenvector perturbation theory up to second-order} \label{appendixA}
Assume that $F=F_0+\varepsilon F_1+\varepsilon^2 F_2 + O(\varepsilon^3 )$ and $F_0$ is diagonal with respect to the basis vectors $w_1,\ldots,w_N$.  We would like to expand the eigenvalues of $F$ in terms of $\varepsilon$. This is a typical problem in perturbative quantum theory \cite{perturbation}. Similar formulas in quantum mechanical perturbation theory can be found in textbooks such as \cite{qmimperial}. The following derivation is reformulated to suit our setting. We focus on the perturbation of degenerate points. Let $f_0$ be a degenerate point of multiplicity $r$ and let  $w_1,\ldots, w_r$ be its associated eigenvectors. Without loss of generality, we assume that $(F_0)_{i,i} = f_0$ for $i=1,\ldots,r$. We define the projection operator
\begin{equation*}
	P:=\begin{pmatrix}
		\text{Id}_r & \\ & 0
	\end{pmatrix}  \quad \text{and let } \ Q:=\text{Id}_N-P.
\end{equation*}

We remark that $F_0$ commutes with $P$ and $Q$. Now, we fix an eigenvector $v_0\in\text{span}\{w_1,\ldots,w_N\}$ and expand $v$ and $f$ as follows
\begin{equation*}
	\begin{split}
		v&=v_0+\varepsilon v_1+\varepsilon^2 v_2+O(\varepsilon^3),\\
		f&=f_0+\varepsilon f_1+\varepsilon^2 f_2+O(\varepsilon^3).
	\end{split}
\end{equation*}
We require first that $v_0 = Pv$, due to the normalization of $v$. From $Fv=fv$,  it follows that up to $O(\varepsilon^2)$ 
\begin{equation}
	\label{Qequation}
	\begin{split}
		& F_0v+\varepsilon(F_1+\varepsilon F_2)v=fv,\\
		& QF_0v+\varepsilon QVv=fQv,\\
		& Q(f \text{Id} -F_0)v=\varepsilon QVv \text{ and } \\
		& Qv =\varepsilon ((f \, \text{Id} -F_0)^{-1}Q)Vv,\end{split}	
\end{equation}
where $V:= F_1 + \varepsilon F_2$. Note that we should treat $((f \, \text{Id} -F_0)^{-1}Q)$ as $0|_{E_{f_0}}\oplus ((f \, \text{Id} -F_0)^{-1}Q)|_{E_{f_0}^c}$, where ${E_{f_0}}$ denotes the eigenspace associated with $f_0$ and $E_{f_0}^c$ is its complementary. Similarly, we obtain that
\begin{equation*}
	PF_0v+\varepsilon PVv = fPv,
\end{equation*}
and therefore,
\begin{equation}
	\label{Pequation}
	f_0Pv+\varepsilon PVv = fPv, 
\end{equation}
where we have used that $PF_0v=F_0Pv=f_0Pv$. Now, we insert $v=Pv+Qv$ into the second term of the left-hand side of (\ref{Pequation}) and derive from $f_0Pv+\varepsilon PV(Pv+Qv)=fPv$ the following two identities:
\begin{equation}
	\label{master}
	\begin{split}
		& f_0Pv+\varepsilon PVPv+\varepsilon PVQv=fPv\text{ and }\\
		& f_0Pv+ \varepsilon PVPv+ \varepsilon^2PV\left((f \, \text{Id} -F_0)^{-1}Q\right)Vv=fPv.
	\end{split}
\end{equation}
For the $\varepsilon^2$-term, we evaluate the expression at $\varepsilon=0$:
\begin{equation}
	\label{hocus}
	PV((f\, \text{Id} -F_0)^{-1}Q)Vv|_{\varepsilon=0}=PF_1\left((f_0 \, \text{Id} -F_0)^{-1}Q\right)F_1v_0:=PF_1GF_1v_0,
\end{equation}
where $G:=(f_0 \, \text{Id} -F_0)^{-1}Q=\text{diag}(0,\ldots,0,(f_0-\lambda_2)^{-1},\ldots,(f_0-\lambda_k)^{-1})$ if we assume that $F_0=\text{diag}(f_0,\ldots,f_0,\lambda_2,\ldots,\lambda_k)$. Hence, we can write that
\begin{equation}
	\label{effectiveH}	
	P\left(f_0 \, \text{Id} +\varepsilon (F_1+\varepsilon F_2)+\varepsilon^2(PF_1GF_1)\right)Pv_0=fv_0.
\end{equation}
With the so-called effective Hamiltonian:
\begin{equation}
	\label{effective}
	\mathcal{H}:=Pf_0P+\varepsilon PF_1P +\varepsilon^2P(F_1GF_1+F_2)P	,
\end{equation}
we can obtain $r$ perturbed eigenvalues up to order $\varepsilon^2$, if we know the form of $F_1$ and $F_2$.\\
Next, we derive the eigenvector perturbation. To that end, we seek a linear combination  $v_0=\beta_1 w_1+\ldots+\beta_2 w_r$ such that equation \eqref{effectiveH} holds. We can read off from the first-order in $\varepsilon$ to derive:
\begin{equation}
    \label{firsteigenvector}
    PF_1P v_0 =f_1v_0.
\end{equation}
For the case where $f_1=0$, we look at the second-order to obtain:
\begin{equation}
    \label{secondeigenvector}
    P(F_1GF_1+F_2)Pv_0 = f_2v_0.
\end{equation}
Furthermore, coming back to the equation:
\begin{equation}
\label{eq:expansion}
    (F_0 + \varepsilon F_1 + \varepsilon^2 F_2 )(v_0 + \varepsilon v_1 + \varepsilon^2 v_2 )=(f_0 + \varepsilon f_1 +\varepsilon^2 f_2) (v_0 + \varepsilon v_1 + \varepsilon^2 v_2)+ O(\varepsilon^3),
\end{equation}
we can derive from the first-order term that
\begin{equation}
        F_0v_1+F_1v_0=f_0v_1+f_1v_0,\\
\end{equation}
which is equivalent to 
\begin{equation}
\label{eq:firstorder}
    (F_0-f_0\mathrm{Id})v_1 = - (F_1-f_1\mathrm{Id})v_0.
\end{equation}
This relates $v_1$ and $v_0$ as follows
\begin{equation}
\label{eq:firstordervector}
    v_1 = G(f_1\mathrm{Id}-F_1)v_0.
\end{equation}

\bibliographystyle{abbrv}
\bibliography{paper_edge}

\begin{thebibliography}{10}

\bibitem{paper1}
H.~Ammari, J.~Cao, and E.~Hiltunen.
\newblock Nonreciprocal wave propagation in space-time modulated media.
\newblock {\em Multiscale Model. Simul., to appear (arXiv preprint
  2109.07220)}, 2022.

\bibitem{davies2019fully}
H.~Ammari and B.~Davies.
\newblock A fully coupled subwavelength resonance approach to filtering
  auditory signals.
\newblock {\em Proc. R. Soc. A}, 475(2228):20190049, 2019.

\bibitem{ammari2021functional}
H.~Ammari, B.~Davies, and E.~O. Hiltunen.
\newblock Functional analytic methods for discrete approximations of
  subwavelength resonator systems, 2021.

\bibitem{ammari2020exceptional}
H.~Ammari, B.~Davies, E.~O. Hiltunen, H.~Lee, and S.~Yu.
\newblock Exceptional points in parity--time-symmetric subwavelength
  metamaterials.
\newblock {\em arXiv preprint arXiv:2003.07796}, 2020.

\bibitem{ammari2020highorder}
H.~Ammari, B.~Davies, E.~O. Hiltunen, H.~Lee, and S.~Yu.
\newblock High-order exceptional points and enhanced sensing in subwavelength
  resonator arrays.
\newblock {\em Studies in Applied Mathematics}, 146(2):440--462, 2021.

\bibitem{ammari2020topological}
H.~Ammari, B.~Davies, E.~O. Hiltunen, and S.~Yu.
\newblock Topologically protected edge modes in one-dimensional chains of
  subwavelength resonators.
\newblock {\em J. Math. Pures Appl.}, 144:17--49, 2020.

\bibitem{fiorani}
H.~Ammari, F.~Fiorani, and H.~E. Orvehed.
\newblock On the validity of the tight-binding method for describing systems of
  subwavelength resonators.
\newblock {\em SIAM J. Appl. Math., to appear}, 2022.

\bibitem{ammari2020honeycomb}
H.~Ammari, B.~Fitzpatrick, E.~O. Hiltunen, H.~Lee, and S.~Yu.
\newblock Honeycomb-lattice minnaert bubbles.
\newblock {\em SIAM Journal on Mathematical Analysis}, 52(6):5441--5466, 2020.

\bibitem{MaCMiPaP}
H.~Ammari, B.~Fitzpatrick, H.~Kang, M.~Ruiz, S.~Yu, and H.~Zhang.
\newblock {\em Mathematical and Computational Methods in Photonics and
  Phononics}, volume 235 of {\em Mathematical Surveys and Monographs}.
\newblock American Mathematical Society, Providence, 2018.

\bibitem{ammari2017subwavelength}
H.~Ammari, B.~Fitzpatrick, H.~Lee, S.~Yu, and H.~Zhang.
\newblock Subwavelength phononic bandgap opening in bubbly media.
\newblock {\em Journal of Differential Equations}, 263(9):5610--5629, 2017.

\bibitem{ammari2017double}
H.~Ammari, B.~Fitzpatrick, H.~Lee, S.~Yu, and H.~Zhang.
\newblock Double-negative acoustic metamaterials.
\newblock {\em Quart. Appl. Math.}, 77(4):767--791, 2019.

\bibitem{ammari2020edge}
H.~Ammari and E.~O. Hiltunen.
\newblock Edge modes in active systems of subwavelength resonators.
\newblock {\em arXiv preprint arXiv:2006.05719}, 2020.

\bibitem{ammari2020time}
H.~Ammari and E.~O. Hiltunen.
\newblock Time-dependent high-contrast subwavelength resonators.
\newblock {\em J. Comp. Phys.}, 445:110594, 2021.

\bibitem{TheaThesis}
H.~Ammari, E.~O. Hiltunen, and T.~Kosche.
\newblock Asymptotic floquet theory for first order odes with finite fourier
  series perturbation and applications in time-modulated metamaterials.
\newblock {\em J. Diff. Equat.}, 319:227--287, 2022.

\bibitem{ammari2020highfrequency}
H.~Ammari, E.~O. Hiltunen, and S.~Yu.
\newblock A high-frequency homogenization approach near the {D}irac points in
  bubbly honeycomb crystals.
\newblock {\em Arch. Ration. Mech. Anal.}, 238(3):1559--1583, 2020.

\bibitem{Thea2}
H.~Ammari and T.~Kosche.
\newblock Homotopical classification of floquet metamaterials.
\newblock {\em arXiv preprint}, 2022.

\bibitem{valley4}
H.~Chen, L.~Yao, H.~Nassar, and G.~Huang.
\newblock Mechanical quantum hall effect in time-modulated elastic materials.
\newblock {\em Phys. Rev. Appl.}, 11:044029, 2019.

\bibitem{bryn_thesis}
B.~Davies.
\newblock {\em Mathematical analysis of subwavelength metamaterials: sensors,
  waveguides and biomimicry}.
\newblock PhD thesis, ETH Z\"urich, 2021.

\bibitem{drouot2}
A.~Drouot.
\newblock The bulk-edge correspondence for continuous dislocated systems.
\newblock {\em arXiv preprint arXiv:1810.10603}, 2018.

\bibitem{drouot1}
A.~Drouot, C.~L. Fefferman, and M.~I. Weinstein.
\newblock Defect modes for dislocated periodic media.
\newblock {\em arXiv preprint arXiv:1810.05875}, 2018.

\bibitem{drouot_valley}
A.~Drouot and M.~Weinstein.
\newblock Edge states and the valley hall effect.
\newblock {\em Adv. Math.}, 368:107142, 2020.

\bibitem{fefferman}
C.~L. Fefferman, J.~P. Lee-Thorp, and M.~I. Weinstein.
\newblock Topologically protected states in one-dimensional continuous systems
  and dirac points.
\newblock {\em P. Nat. Acad. Sci. USA}, 111(24):8759--8763, 2014.

\bibitem{fefferman2}
C.~L. Fefferman, J.~P. Lee-Thorp, and M.~I. Weinstein.
\newblock Edge states in honeycomb structures.
\newblock {\em Ann. PDE}, 2(2):12, Dec 2016.

\bibitem{fefferman3}
C.~L. Fefferman, J.~P. Lee-Thorp, and M.~I. Weinstein.
\newblock Honeycomb schrödinger operators in the strong binding regime.
\newblock {\em Commun. Pure Appl. Math.}, 71(6):1178--1270, 2018.

\bibitem{fefferman4}
C.~L. Fefferman and M.~I. Weinstein.
\newblock Honeycomb lattice potentials and dirac points.
\newblock {\em J. Am. Math. Soc.}, 25(4):1169–1220, 2012.

\bibitem{feppon}
F.~Feppon and H.~Ammari.
\newblock Modal decompositions and point scatterer approximations near the
  minnaert resonance frequencies.
\newblock {\em Studies Appl. Math., DOI: 10.1111/sapm.12493}, 2022.

\bibitem{fleury2016floquet}
R.~Fleury, A.~B. Khanikaev, and A.~Al{\`u}.
\newblock Floquet topological insulators for sound.
\newblock {\em Nature communications}, 7(1):1--11, 2016.

\bibitem{review2}
H.~Ge, M.~Yang, C.~Ma, M.-H. Lu, Y.-F. Chen, N.~Fang, and P.~Sheng.
\newblock Breaking the barriers: advances in acoustic functional materials.
\newblock {\em National Science Review}, 5:159--182, 2018.

\bibitem{haldane2}
F.~Haldane and S.~Raghu.
\newblock Possible realization of directional optical waveguides in photonic
  crystals with broken time-reversal symmetry.
\newblock {\em Phys. Rev. Lett.}, 100(1):013904, 2008.

\bibitem{topo1}
L.~He, Z.~Addison, J.~Jin, E.~J. Mele, S.~G. Johnson, and B.~Zhen.
\newblock Floquet chern insulators of light.
\newblock {\em Nature Communications}, 10:4194, 2019.

\bibitem{erik_thesis}
E.~Hiltunen.
\newblock {\em Asymptotic analysis of high-contrast subwavelength resonator
  structures}.
\newblock PhD thesis, ETH Z\"urich, 2021.

\bibitem{valley6}
R.~Kumar~Pal and M.~Ruzzene.
\newblock Edge waves in plates with resonators: an elastic analogue of the
  quantum valley hall effect.
\newblock {\em New J. Phys.}, 19:025001, 2017.

\bibitem{lee-thorp}
J.~P. Lee-Thorp, M.~I. Weinstein, and Y.~Zhu.
\newblock Elliptic operators with honeycomb symmetry: Dirac points, edge states
  and applications to photonic graphene.
\newblock {\em Arch. Ration. Mech. An.}, 232(1):1--63, Apr 2019.

\bibitem{lemoult2016soda}
F.~Lemoult, N.~Kaina, M.~Fink, and G.~Lerosey.
\newblock Soda cans metamaterial: A subwavelength-scaled phononic crystal.
\newblock {\em Crystals}, 6(7), 2016.

\bibitem{valley7}
M.~Li, X.~Ni, M.~Weiner, A.~Al{\`u}, and A.~B. Khanikaev.
\newblock Topological phases and nonreciprocal edge states in non-hermitian
  floquet insulators.
\newblock {\em Phys. Rev. B}, 100:045423, 2019.

\bibitem{valley3}
J.~Liu, Z.~Ma, J.~Gao, and X.~Dai.
\newblock Quantum valley hall effect, orbital magnetism, and anomalous hall
  effect in twisted multilayer graphene systems.
\newblock {\em Phys. Rev. X}, 9:031021, 2019.

\bibitem{phononic2}
Z.~Liu, C.~Chan, and P.~Sheng.
\newblock Analytic model of phononic crystals with local resonances.
\newblock {\em Phys. review B}, 71(1):014103, 2005.

\bibitem{phononic1}
Z.~Liu, X.~Zhang, Y.~Mao, Y.~Zhu, Z.~Yang, C.~Chan, and P.~Sheng.
\newblock Locally resonant sonic materials.
\newblock {\em Science}, 289(5485):1734--1736, 2000.

\bibitem{review}
G.~Ma and P.~Sheng.
\newblock Acoustic metamaterials: From local resonances to broad horizons.
\newblock {\em Science Advances}, 2(2):e1501595, 2016.

\bibitem{valley1}
J.~Noh, S.~Huang, K.~Chen, and M.~Rechtsman.
\newblock Observation of photonic topological valley hall edge states.
\newblock {\em Phys. Rev. Lett.}, 120:063902, 2018.

\bibitem{haldane}
S.~Raghu and F.~Haldane.
\newblock Analogs of quantum-hall-effect edge states in photonic crystals.
\newblock {\em Phys. Rev. A}, 78:033834, 2008.

\bibitem{qmimperial}
B.~Ryan.
\newblock Lecture notes for quantum mechanics ii.
\newblock pages 43--52, 2020.

\bibitem{perturbation}
B.~Simon.
\newblock Large orders and summability of eigenvalue perturbation theory: A
  mathematical overview.
\newblock {\em International Journal of Quantum Chemistry}, 21(1):3--25, 1982.

\bibitem{teschl2012ordinary}
G.~Teschl.
\newblock {\em Ordinary Differential Equations and Dynamical Systems}.
\newblock Graduate studies in mathematics. American Mathematical Society, 2012.

\bibitem{wang2019subwavelength}
L.~Wang, R.-Y. Zhang, B.~Hou, Y.~Huang, S.~Li, and W.~Wen.
\newblock Subwavelength topological edge states based on localized spoof
  surface plasmonic metaparticle arrays.
\newblock {\em Opt. Express}, 27(10):14407--14422, May 2019.

\bibitem{Yakubovich}
V.~Yakubovich and V.~Starzhinskii.
\newblock {\em Linear differential equations with periodic coefficients},
  volume 1,2.
\newblock John Wiley $\&$ Sons, 1975.

\bibitem{yves2017crytalline}
S.~Yves, R.~Fleury, T.~Berthelot, M.~Fink, F.~Lemoult, and G.~Lerosey.
\newblock Crystalline metamaterials for topological properties at subwavelength
  scales.
\newblock {\em Nat. Commun.}, 8(1):16023, Jul 2017.

\bibitem{yves2017crystalline}
S.~Yves, R.~Fleury, T.~Berthelot, M.~Fink, F.~Lemoult, and G.~Lerosey.
\newblock Crystalline metamaterials for topological properties at subwavelength
  scales.
\newblock {\em Nat. Commun.}, 8:16023 EP --, Jul 2017.
\newblock Article.

\bibitem{yves2017topological}
S.~Yves, R.~Fleury, F.~Lemoult, M.~Fink, and G.~Lerosey.
\newblock Topological acoustic polaritons: robust sound manipulation at the
  subwavelength scale.
\newblock {\em New J. Phys.}, 19(7):075003, 2017.

\bibitem{valley2}
S.~Yves, G.~Lerosey, and F.~Lemoult.
\newblock Structure-composition correspondence in crystalline metamaterials for
  acoustic valley-hall effect and unidirectional sound guiding.
\newblock {\em EPL}, 129:44001, 2020.

\end{thebibliography}
\end{document}